\documentclass[reqno]{amsart}

\usepackage{amssymb}
\usepackage{amsfonts}
\usepackage{mathrsfs}
\usepackage{bbm}

\usepackage{xcolor}

\theoremstyle{plain}
    \newtheorem{theorem}{Theorem}[section]
    \newtheorem*{theorem*}{Theorem}
    \newtheorem{lemma}[theorem]{Lemma}
    \newtheorem{proposition}[theorem]{Proposition}
    \newtheorem{corollary}[theorem]{Corollary}

\theoremstyle{definition}
    \newtheorem{definition}{Definition}[section]
    \newtheorem{remark}{Remark}[section]
    
\numberwithin{equation}{section}

\renewcommand{\r}{\right}
\DeclareMathOperator{\Span}{span}

\begin{document}

\title[]{Long-time asymptotics of the damped nonlinear Klein-Gordon equation with a delta potential}
\author[K. Ishizuka]{Kenjiro Ishizuka}
\address[K.Ishizuka]{Research Institute for Mathematical Sciences, Kyoto University, Kyoto 606-8502, JAPAN}
\email{ishizuka@kurims.kyoto-u.ac.jp}
\date{\today}
\date{\today}
\maketitle

\begin{abstract}
We consider the damped nonlinear Klein-Gordon equation with a delta potential
\begin{align*}
\partial_{t}^2u-\partial_{x}^2u+2\alpha \partial_{t}u+u-\gamma {\delta}_0u-|u|^{p-1}u=0, \ & (t,x) \in \mathbb{R} \times \mathbb{R},
\end{align*}
where $p>2$, $\alpha>0,\ \gamma<2$, and $\delta_0=\delta_0 (x)$ denotes the Dirac delta with the mass at the origin. When $\gamma=0$,  C\^{o}te, Martel and Yuan \cite{CMY} proved that any global solution either converges to 0 or to the sum of $K\geq 1$ decoupled solitary waves which have alternative signs. In this paper, we first prove that any global solution either converges to 0 or to the sum of $K\geq 1$ decoupled solitary waves. Next we construct a single solitary wave solution that moves away from the origin when $\gamma<0$ and construct an even 2-solitary wave solution when $\gamma\leq -2$. Last we give single solitary wave solutions and even 2-solitary wave solutions an upper bound for the distance between the origin and the solitary wave.
\end{abstract}

\tableofcontents

\section{Introduction}
\subsection{Setting of the problem}
We consider the following damped nonlinear Klein-Gordon equation with a delta potential
\begin{align}
\label{DNKG}
\tag{DNKG}
	\left\{
	\begin{aligned}
		&\partial_{t}^2u-\partial_{x}^2u+2\alpha \partial_{t}u+u-\gamma {\delta}_0u-f(u)=0, & (t,x) \in \mathbb{R} \times \mathbb{R},
		\\
		&u(0,x)=u_{0}(x), &x \in \mathbb{R},
	\end{aligned}
	\r.
\end{align}
where $f(u)=|u|^{p-1}u$, $p>2$, $\alpha>0,\ \gamma<2$, and $\delta_0=\delta_0 (x)$ denotes the Dirac delta with the mass at the origin. 
We define the operator $\mathcal{A}$ as 
\begin{align*}
	\mathcal{A}=\begin{pmatrix}
0 & 1  \\
{\partial}_{x}^2-1 & -2\alpha   \\
\end{pmatrix}
,
\end{align*}
with domain 
\begin{align}\label{Adomain}
 \mathcal{D}=\{ (u,v) \in H^{2,\ast}\times H^1 : u^{\prime}(0^+)-u^{\prime}(0^-)=-\gamma u(0)\} \subset \mathcal{H},
\end{align}
where 
\begin{align*}
H^{2,\ast}&=H^1(\mathbb{R})\cap H^2(\mathbb{R}\setminus \{ 0\}),
 \\
\mathcal{H}&=H^1(\mathbb{R})\times L^2(\mathbb{R}).
\end{align*}
$\mathcal{A}$ generates a strongly continuous semigroup on $\mathcal{H}$, which allows us to establish the local well-posedness of \eqref{DNKG}. We will show it in Section 2. 
Moreover defining the energy of a solution $\vec{u}=(u, {\partial}_tu)$ by
\begin{align*}
	E_{\gamma}(\vec{u})=\frac{1}{2}\left(\| u\|_{H^1}^2+\|{\partial}_tu\|_{L^2}^2-\gamma|u(0)|^2\right)-\frac{1}{p+1}\| u\|_{L^{p+1}}^{p+1},
\end{align*}
it holds 
\begin{align}\label{energydecay}
	E_{\gamma}(\vec{u}(t_2))-E_{\gamma}(\vec{u}(t_1))=-2\alpha \int_{t_1}^{t_2} \| {\partial}_tu(t)\|_{L^2}^2dt.
\end{align}

Furthermore the equation \eqref{DNKG} enjoys several invariances:
\begin{align}\label{invariance}
\begin{aligned}
u(t,x)&\mapsto -u(t,x),
\\
u(t,x)&\mapsto u(t,-x).
\end{aligned}
\end{align}

When $\gamma=0$, it is well-known that up to sign and translation, the only stationary solution of \eqref{DNKG}  is $(Q,0)$, where $Q$ is 
\begin{align*}
	Q(x)=\left(\frac{p+1}{2\cosh^{2}\big(\frac{p-1}{2}x\big)}\right)^{\frac{1}{p-1}},
\end{align*}
which solves the equation 
\begin{align*}
	Q^{\prime \prime }-Q+Q^p=0 \ \ \mbox{on} \ \ \mathbb{R}
\end{align*}
(see \cite{K2}). From the expression of $Q$ , we obtain that as $|x|\to \infty$,
\begin{align}\label{Qes}
\left|Q(x)-c_Qe^{-|x|}\right|\lesssim e^{-2|x|},\ \left|Q^{\prime}(x)+\frac{x}{|x|}c_Qe^{-|x|}\right|\lesssim e^{-2|x|},
\end{align}
where $c_Q=(2p+2)^{\frac{1}{p-1}}$. When $\gamma=0$,  it is proved in \cite{CMY} that any global solution either converges to 0  or to the sum of $K\geq 1$ decoupled solitary waves which have alternative signs, namely for any global solution $\vec{u} \in C([0,\infty),\mathcal{H})$ of \eqref{DNKG} there exist $K\geq0,  \sigma=\pm 1$ and 
functions $z_k:[0,\infty)\to \mathbb{R}$\ for all $k=1,\cdots, K$ such that  
\begin{align}\label{dnkgsr}
\begin{aligned}
	&\lim_{t\to \infty} \left(\| u(t)-\sigma \sum_{k=1}^K(-1)^kQ(\cdot -z_k(t))\|_{H^1}+\| {\partial}_tu(t)
	\|_{L^2}\right)=0,
	\\
	&\lim_{t\to \infty} \left(z_{k+1}(t)-z_k(t)\right)=\infty\ for\ k=1,2,\cdots,K-1 .
\end{aligned}
\end{align}

\begin{remark}
If $K=0$ then the terms $\sum_{k=1}^K$ are $0$.
\end{remark}

On the other hand, when $\gamma \neq 0$, there may exist a solution of \eqref{DNKG} whose asymptotic behavior is different from the $\gamma=0$ case thanks to the potential-soliton repulsion(or attraction).
In fact, when $\gamma$ satisfies $|\gamma|<2$ and $\gamma\neq0$, the positive stationary solution of \eqref{DNKG} is $(Q_{\gamma},0)$, where $Q_{\gamma}$ is 
\begin{align*}
	Q_{\gamma}(x)=\left( \frac{p+1}{2{\cosh}^2\left( \frac{p-1}{2}|x|+{\tanh}^{-1}(\frac{\gamma}{2} )\right) }\right)^{\frac{1}{p-1}},
\end{align*}
which is the solution of 
\begin{align*}
Q_{\gamma}^{\prime \prime}-Q_{\gamma}+\gamma{\delta}_0Q_{\gamma}+Q_{\gamma}^p=0.
\end{align*}

Its asymptotic behavior is different from \eqref{dnkgsr}.

\subsection{Main results}

We are interested in the long-time asymptotic behavior of global solutions of \eqref{DNKG}. When $\gamma \neq 0$, it is difficult to describe precisely the asymptotic behavior as $t\to \infty$ of all global solutions of \eqref{DNKG} as \cite{CMY} because the potential-soliton repulsion (or attraction) is delicate.  However by using the concentration-compactness arguments we obtain the soliton resolution as the following:

\begin{theorem}\label{sr}
For any global solution $\vec{u}$ of \eqref{DNKG}, there exist $K\geq 0$, signs $\sigma=0,\pm1$, $\sigma_k=\pm 1$ for any $k\in \{ 1,2,\cdots,K\}$, and functions $z_k:[0,\infty)\to \mathbb{R}$ for any $k\in \{ 1,2,\cdots,K\}$ such that   
\begin{align}\label{sr-1}
\begin{aligned}
\lim_{t\to \infty} &\left(\|u(t)-\sigma Q_{\gamma}-\sum_{k=1}^K{\sigma}_kQ(\cdot-z_k(t))\|_{H^1}+\|{\partial}_tu(t)\|_{L^2}\right)=0,
\\
\lim_{t\to \infty}&(z_{k+1}(t)-z_k(t))=\infty\ \mbox{for}\ k=1,\cdots,K-1,
\\
\lim_{t\to \infty}&|z_k(t)|=\infty\ \mbox{for}\ k=1,\cdots,K.
\end{aligned}
\end{align}
\end{theorem}
\begin{remark}
When $\sigma=0$ and $K=0$, $\vec{u}$ converges exponentially to $0$ in $\mathcal{H}$ as $t\to \infty$.  We prove this in Section 4.2.
\end{remark}

\begin{remark}
If $\gamma\leq -2$, we may assume $\sigma=0$ because $Q_{\gamma}$ does not exist. 
\end{remark}
Theorem \ref{sr} implies that \eqref{sr-1} exhausts all solutions in the energy space. On the other hand, it is difficult to determine signs $(\sigma_k)_{1\leq k\leq K}$ of \eqref{sr-1}. For example, we do not know whether there exists a global solution $\vec{u}$ of \eqref{DNKG} satisfying \eqref{sr-1} with $K=3$ and $(\sigma,\sigma_1,\sigma_2,\sigma_3)=(0,1,1,-1)$. This is due to the potential-soliton repulsion (attraction). When $\gamma=0$, we only consider the interaction between solitary waves and \eqref{dnkgsr} exhausts all solutions. On the other hand, when $\gamma\neq 0$ we need to reveal whether the interaction between solitary waves is stronger than the potential-soliton repulsion (attraction). However, the potential-soliton repulsion (attraction) is delicate and it prevents from determining signs of solitary waves.

\begin{remark}
By \cite{CMY} and \eqref{Adomain}, for any $K\in \mathbb{N}$ and $\sigma=\pm1$ there exist functions $z_k:[0,\infty)\to \mathbb{R}$ for $k=1,2,\cdots,K$ and an odd global solution $\vec{u} \in C([0,\infty),\mathcal{H})$ of \eqref{DNKG} such that
\begin{align*}
&\lim_{t\to \infty} \left(\| u(t)-\sigma \sum_{k=1}^{K}\left\{Q(\cdot-z_k(t))-Q(\cdot+z_k(t))\right\}\|_{H^1}+\|{\partial}_tu(t)\|_{L^2}\right)=0,
\\
&\lim_{t\to\infty}z_k(t)=\infty\ for\ k=1,2,\cdots,K,
\\
& \lim_{t\to \infty} (z_{k+1}(t)-z_k(t))=\infty\ for\ k=1,2,\cdots,K-1.
\end{align*}
\end{remark}

As above, it is difficult to determine signs $\sigma$ and $(\sigma_k)_{1\leq k\leq K}$ of \eqref{sr-1}. On the other hand, when \eqref{sr-1} with $\sigma=0$, $K=1$ and $\gamma<0$ or \eqref{sr-1} with $\sigma=0$, $K=2$ and $\gamma\leq -2$, thanks to the structure of the ground state energy we can construct a global solution whose asymptotic behavior is different from \eqref{dnkgsr}.

\begin{theorem}\label{existencesoliton}
\begin{enumerate}
\item Let $\gamma<0$. For $\sigma=\pm1$, there exists a global solution $\vec{u}\in C([0,\infty); \mathcal{H})$ of \eqref{DNKG} such that
\begin{align}
\label{1soli}
\begin{aligned}
\lim_{t\to \infty} \left(\|u(t)-\sigma Q(\cdot-z(t))\|_{H^1}+\|{\partial}_tu(t)\|_{L^2}\right)&=0,
\\
\lim_{t\to\infty}|z(t)|&=\infty,
\end{aligned}
\end{align}
for some $z:[0,\infty)\to \mathbb{R}$.
\item Let $\gamma\leq -2$. For $\sigma=\pm1$, there exists an even global solution $\vec{u}\in C([0,\infty); \mathcal{H})$ of \eqref{DNKG} such that
\begin{align}
\label{2soli}
\begin{aligned}
\lim_{t\to \infty} \left(\|u(t)-\sigma(Q(\cdot-z(t))+Q(\cdot+z(t)))\|_{H^1}+\|{\partial}_tu(t)\|_{L^2}\right)&=0,
\\
\lim_{t\to\infty}z(t)&=\infty,
\end{aligned}
\end{align}
for some $z:[0,\infty)\to \mathbb{R}$.
\end{enumerate}
\end{theorem}
\begin{remark}
When $\gamma=0$, a solution $\vec{u}$ of \eqref{DNKG} satisfying \eqref{dnkgsr} with $K=1$ converges to the stationary solution $\pm Q(\cdot-y)$ for some $y\in \mathbb{R}$ as $t\to \infty$. However, when $\gamma<0$, a solution $\vec{u}$ of \eqref{DNKG} satisfying \eqref{1soli}  behaves asymptotically as $t\to \infty$ so that $\pm Q$ moves far away from the origin.
\end{remark}

To establish our next result, we define a solution as \eqref{1soli} or \eqref{2soli}.
\begin{definition} 
\begin{enumerate}
\item A solution $\vec{u}$ of \eqref{DNKG} is called a \textit{single solitary wave} if there exist a sign $\sigma=\pm1$ and a function $z:[0,\infty)\to \mathbb{R}$ such that \eqref{1soli} is satisfied.
\item A solution $\vec{u}$ of \eqref{DNKG} is called an \textit{even 2-solitary wave} if $\vec{u}$ is even and there exist a sign $\sigma=\pm1$ and a function $z:[0,\infty)\to \mathbb{R}$ such that \eqref{2soli} is satisfied.
\end{enumerate}
\end{definition}

The asymptotic behavior of a single solitary wave appears as $Q$ moving away from the origin. Furthermore, an even 2-solitary wave asymptotes to the superposition of two solitary waves of the same sign,  which implies the potential-soliton repulsion is stronger than the soliton-soliton attraction. Our main result is that if a solution $\vec{u}$ of \eqref{DNKG} is a single solitary wave or an even 2-solitary wave, the distance between the center of the solitary wave and the origin has an upper bound.

\begin{theorem}\label{centerdistance}
\begin{enumerate}
\item Let $\gamma<0$. If a global solution $\vec{u}$ is a single solitary wave, then we have
\begin{align*}
|z(t)|-\frac{1}{2}\log{t}\lesssim 1.
\end{align*}

\item Let $\gamma\leq -2$. If a global solution $\vec{u}$ is an even 2-solitary wave, then we have
\begin{align*}
z(t)-\frac{1}{2}\log{t}\lesssim 1.
\end{align*}
\end{enumerate}
\end{theorem}
Gathering Theorem \ref{existencesoliton} and Theorem \ref{centerdistance}, if $\gamma<0$ there exists a solution of \eqref{DNKG} satisfying \eqref{1soli} by the potential-soliton repulsion, but the potential-soliton repulsion is not strong so that the distance between origin and the center of the solitary wave is roughly less than $\frac{1}{2}\log{t}$. Furthermore, if $\gamma\leq -2$, there exists a soliton of \eqref{DNKG} satisfying \eqref{2soli} by the potential-soliton repulsion, which means that the potential-soliton repulsion is stronger than the soliton-soliton attraction.

\subsection{Previous results}
There have been intensive study on the global behavior of general (large) solutions for the nonlinear dispersive equations, where the guiding principle is \textit{soliton resolution conjecture}, which claims that generic global solutions are asymptotic to the superposition of solitary waves for large time. In particular, the conjecture has been studied for the energy-critical wave equation. We refer to \cite{CDKM,DJKM,DKM1,DKM2,JL4}, where the conjecture has been proved for all radial solutions in space dimension $N=3,5,7,\cdots$ or $N=4,6$. Similar results have been obtained in the case with a potential \cite{JLX,LMZ}, with a damping term \cite{GZ}, as well as for the wave maps \cite{DKMM,JL1}, all under some rotational symmetry.

On the other hand, the conjecture is widely open for the nonlinear Klein-Gordon equation, namely the $\alpha=0$ and $\gamma=0$ case. Nakanishi and Schlag \cite{NS} have proved the conjecture as long as the energy of the solution is not much larger than the ground state energy in space dimension $N=3$. 

However, for the damped nonlinear Klein-Gordon equation, namely the $\gamma=0$ case, the energy decay \eqref{energydecay} makes the analysis simpler than the undamped case. Historically Keller \cite{K1} constructed stable and unstable manifolds in a neighborhood of any stationary solution of \eqref{DNKG}, and Feireisl \cite{F} proved the conjecture along a time sequence. Burq, Raugel and Schlag \cite{BRS} proved the conjecture in the full limit $t\to\infty$ for all radial solutions. More recently, C\^{o}te, Martel and Yuan \cite{CMY} proved the conjecture without size restriction and radial assumption. To the best of our knowledge, this is the first complete proof of the soliton resolution conjecture for a non-integrable equation without any size or symmetry restriction and with moving solitons. Other results for the damped nonlinear Klein-Gordon equation are in \cite{CMYZ,CY,IN}. Their arguments are useful for \eqref{DNKG}.

 In addition, there are many results about multi-solitary waves for the nonlinear dispersive equations. For the nonlinear Klein-Gordon equation, C\^{o}te and Mun\~{o}z \cite{CM} have constructed multi-solitary waves. Note that multi-solitary waves move at a constant speed. More recently, Aryan \cite{A} has constructed 2-solitary waves whose speed converges to $0$ as $t\to\infty$. Strictly in \cite{A} the distance between each solitary wave is $2(1+O(1))\log{t}$ as $t\to \infty$. As \cite{A}, there are several results about multi-solitary waves with logarithmic distance; see \textit{e.g.} \cite{GI,MN,N1,N2}. We note that for the damped nonlinear Klein-Gordon equation, multi-solitary waves of \eqref{DNKG} are impossible to move at a constant speed by the damping $\alpha$. See also \cite{J1,J2,JL1,JL2,JL4,JL3} for works about two-bubble solutions.

There are few results about \eqref{DNKG} with $\alpha=0$ because of the singularity of $\delta_0$. Csobo, Genoud, Ohta and Royer \cite{CGOR} have proved the stability of standing waves for \eqref{DNKG}. On the other hand, there are many results for the nonlinear Schr\"{o}dinger equation with a delta potential; see \textit{e.g.} \cite{FJ,FOO,GI,GIS,II}. Notably, we have several technical tools in \cite{FJ}. We also see the potential-soliton repulsion in \cite{GI}.

\section{Local well-posedness  and uniform bound}
In this section we discuss the local well-posedness of the Cauchy problem \eqref{DNKG}. First we show that the operator $\mathcal{A}$ generates a $C^0$-semigroup on $\mathcal{H}$.  Then \eqref{DNKG} can be rewritten as 
\begin{align}
\label{duhamel}
	\left\{
	\begin{aligned}
		&\frac{d}{dt}U(t)-\mathcal{A}U(t)=F(U(t)),
		\\
		&U(0)=U_0,
	\end{aligned}
	\r.
\end{align}
where $U=(u,v)\in \mathcal{H}$ and 
\begin{align*}
	F(U)=\begin{pmatrix}
0  \\
f(u)   \\
\end{pmatrix}
.
\end{align*}
Thus we only need to show the local well-posedness of \eqref{duhamel}, and this can be proven by the standard arguments (see \cite{CH}).

\subsection{Linear evolution in the energy space}

In this subsection we show that the operator $\mathcal{A}$ generates a strongly continuous group on $\mathcal{H}$. Since $\mathcal{D}$ defined by \eqref{Adomain} is dense in $\mathcal{H}$, we can consider $\mathcal{A}$ on $\mathcal{H}$ (See \cite{CGOR}). Next, for $\mu \geq 0$ we introduce on $\mathcal{H}$ the quadratic form defined by
\begin{align*}
\| (u,v)\|_{\mathcal{H},\mu ,\gamma}=\| u^{\prime}\|_{L^2}^2+\mu^2\|u\|_{L^2}^2-\gamma|u(0)|^2+\|v\|_{L^2}^2.
\end{align*}
 We denote by $\langle \cdot , \cdot \rangle_{\mathcal{H}, \mu , \gamma}$ the corresponding bilinear form. Then by direct calculation for $U=(u,v) \in \mathcal{D}$, 
 \begin{align*}
\langle \mathcal{A}U,U\rangle_{\mathcal{H}, 1,\gamma}=-2\alpha \| v\|_{L^2}^2\leq 0.
\end{align*}
In addition the following lemma holds.
\begin{lemma}\label{deltaequivalent}
For $\gamma<2$ there exists $C_{\gamma}>1$ such that for all $u\in H^1$, we have 
\begin{align}\label{deqes}
C_\gamma^{-1}\| u\|_{H^1}^2\leq \| u\|_{H^1}^2-\gamma|u(0)|^2\leq C_\gamma\|u\|_{H^1}^2.
\end{align}
\end{lemma}
\begin{proof}
See \cite[Lemma 2.3]{CGOR}.
\end{proof}

\begin{proposition}\label{HY}
The operator $\mathcal{A}$ generates a $C^0$-semigroup on $\mathcal{H}$.
\end{proposition}
\begin{proof}
Let $0<\mu<1$ be close to $1$.
 For $U=(u,v)\in \mathcal{D}$, we have 
\begin{align*}
\langle \mathcal{A}U,U\rangle_{\mathcal{H},\mu,\gamma}=(\mu^2-1)\langle u,v\rangle_{L^2}-2\alpha\|v\|_{L^2}^2,
\end{align*}
and 
\begin{align*}
|\langle \mathcal{A}U,U\rangle_{\mathcal{H},\mu,\gamma}|\lesssim \|u\|_{L^2}^2+\|v\|_{L^2}^2.
\end{align*}
Moreover by Lemma \ref{deltaequivalent}, 
\begin{align*}
\| U\|_{\mathcal{H},\mu,\gamma}^2\gtrsim \|u\|_{L^2}^2+\|v\|_{L^2}^2.
\end{align*}
Hence, fixing $\beta>0$ large enough, we have 
\begin{align*}
\langle (\pm \mathcal{A}-\beta)U,U\rangle_{\mathcal{H},\mu,\gamma}\leq 0.
\end{align*}
Therefore the operators $\pm \mathcal{A}-\beta$ are dissipative. In particular, for $\lambda > \beta$ we have
\begin{align}\label{inversebound}
\| (\pm \mathcal{A}-\lambda)U\|_{\mathcal{H},\mu,\gamma}^2\geq \|(\pm \mathcal{A}-\beta)U\|_{\mathcal{H},\mu,\gamma}^2+(\lambda-\beta)^2\|U\|_{\mathcal{H},\mu,\gamma}^2,
\end{align}
so that $\pm \mathcal{A}-\lambda$ are injective with close range. Next we show that $\mathcal{A}-\lambda$ is surjective. Let $F=(f,g)\in \mathcal{H}$. For $U=(u,v) \in \mathcal{D}$, we have
\begin{align*}
(A-\lambda)U=F\Leftrightarrow 
\begin{cases} 
v=\lambda u+f,
\\
u^{\prime \prime}-(\lambda^2+2\alpha\lambda+1)u=g+(\lambda+2\alpha)f.
\end{cases} 
\end{align*}
Then by \cite[Lemma 2.5]{CGOR}, the operator $\mathcal{A}-\lambda$ is surjective. Hence, $\mathcal{A}-\lambda$ has a bounded inverse and by \eqref{inversebound}, 
\begin{align*}
\|(\mathcal{A}-\lambda)^{-1}\|_{\mathcal{L}(\mathcal{H})}\leq \frac{1}{\lambda-\beta}.
\end{align*}
Therefore by the Hille-Yoshida Theorem, this proves that $\mathcal{A}$ generates a $C^0$-semigroup on $\mathcal{H}$. Now the same holds with $\mathcal{A}$ replaced by $-\mathcal{A}$, and the proof is complete.
\end{proof}

Furthermore we estimate the strongly continuous group $(e^{t\mathcal{A}})_{t\geq 0}$.
\begin{proposition}\label{lineares}
There exist $C>0$ and $\kappa>0$ such that, for all $t\geq 0$
\begin{align}\label{le}
\| e^{t\mathcal{A}}\|_{\mathcal{L}(\mathcal{H})}\leq Ce^{-\kappa t},\  \|  e^{t\mathcal{A}}\|_{\mathcal{L}(L^2\times H^{-1})}\leq Ce^{-\kappa t}.
\end{align}
\end{proposition}
\begin{proof}
Consider the linear equation
\begin{align*}
\label{LDNKG}
	\left\{
	\begin{aligned}
		&\partial_{t}^2u-\partial_{x}^2u+2\alpha \partial_{t}u+u-\gamma {\delta}_0u=0, & (t,x) \in \mathbb{R} \times \mathbb{R},
		\\
		&u(0,x)=u_{0}(x), &x \in \mathbb{R},
	\end{aligned}
	\r.
\end{align*}
and define $E$ as
\begin{align*}
E(t)=\frac{1}{2}\left(\|{\partial}_tu(t)\|_{L^2}^2+\|u(t)\|_{H^1}^2-\gamma|u(t,0)|^2\right)+\epsilon\left(\int u(t){\partial}_tu(t)+\alpha \|u(t)\|_{L^2}^2\right),
\end{align*} 
where $\epsilon>0$ small enough. Then by direct calculation, $E$ satisfies
\begin{align*}
  E^{\prime}(t) &= -2\alpha\|{\partial}_tu(t)\|_{L^2}^2+\epsilon\left( \|{\partial}_tu(t)\|_{L^2}^2-\|u(t)\|_{H^1}^2+\gamma|u(t,0)|^2\right)
  \\
   & =-\frac{\epsilon}{2}\left(\|{\partial}_tu(t)\|_{L^2}^2+\|u(t)\|_{H^1}^2-\gamma|u(t,0)|^2\right)-\frac{\epsilon}{2}\left(\|u(t)\|_{H^1}^2-\gamma|u(t,0)|^2\right)
   \\
   &\quad -(2\alpha-\frac{5}{4}\epsilon)\|{\partial}_tu(t)\|_{L^2}^2.
\end{align*}
Since $\epsilon$ is small enough, we have
\begin{align*}
&\quad \epsilon^2\left(\int u(t)\partial_tu(t)+\alpha\|u(t)\|_{L^2}^2\right)
\\
&\leq (\alpha+\frac{1}{2})\epsilon^2\|u(t)\|_{H^1}^2+\frac{\epsilon^2}{2}\|{\partial}_tu(t)\|_{L^2}^2
\\
&\leq (\alpha+\frac{1}{2})C_{\gamma}^{-1}\epsilon^2\left(\|u(t)\|_{H^1}^2-\gamma|u(t,0)|^2\right)+\frac{\epsilon^2}{2}\|{\partial}_tu(t)\|_{L^2}^2
\\
&\leq \frac{\epsilon}{2}\left(\|u(t)\|_{H^1}^2-\gamma|u(t,0)|^2\right)+(2\alpha-\frac{5}{4}\epsilon)\|{\partial}_tu(t)\|_{L^2}^2.
\end{align*}
Therefore we have
\begin{align*}
E^{\prime}(t)<-\epsilon E(t),
\end{align*}
which implies 
\begin{align*}
E(t)<e^{-\epsilon t}E(0).
\end{align*}
Furthermore, we define ${\gamma}^{\prime}=\frac{\gamma}{1-\epsilon}$. Then $\gamma<{\gamma}^{\prime}<2$ holds because $\epsilon$ is small enough.  Therefore we have 
\begin{align*}
E(t)&\geq \frac{1}{2}(1-\epsilon)\left(\|u(t)\|_{H^1}^2-\frac{\gamma}{1-\epsilon }|u(t,0)|^2\right)+(\frac{1}{2}-\frac{\epsilon}{2})\|{\partial}_tu(t)\|_{L^2}^2
\\
&\geq\frac{1}{2}(1-\epsilon)C_{{\gamma}^{\prime}}^{-1}\|u(t)\|_{H^1}^2+(\frac{1}{2}-\frac{\epsilon}{2})\|{\partial}_tu(t)\|_{L^2}^2.
\end{align*}
So we get $E(t)\sim \|\vec{u}(t)\|_{\mathcal{H}}^2$, and 
\begin{align*}
\| \vec{u}(t)\|_{\mathcal{H}}^2\lesssim e^{-\epsilon t}\|\vec{u}(0)\|_{\mathcal{H}}^2.
\end{align*}
This means that $\| e^{t\mathcal{A}}\|_{\mathcal{L}(\mathcal{H})}\lesssim e^{-\kappa t}$ for some $\kappa>0$. Now the same holds with $\mathcal{H}$ replaced by $L^2\times H^{-1}$, and the proof is complete.

\end{proof}

\subsection{Proof of Local well-posedness of the Cauchy problem}
We now turn to the equation \eqref{duhamel}. We show that the equation \eqref{duhamel} is locally well-posed in $\mathcal{H}$ by using the standard arguments as \cite{CH}. 
\begin{proposition}\label{lwp}
The following properties hold.
\begin{enumerate}
\item For any $U_0\in \mathcal{H}$, the equation \eqref{duhamel} has a unique maximal solution $U\in C([0,T(U_0)), \mathcal{H})$, with initial data $U_0$.

\item If $T(U_0)<\infty$, then $\lim_{t\to T(U_0)-0}\|U(t)\|_{\mathcal{H}}=\infty$.

\item $T:\mathcal{H}\to (0,\infty]$ is lower semicontinuous: given the initial conditions $U_0,U_{0,n}\in \mathcal{H}$ such that $U_{0,n}$ converges to $U_0$ in $\mathcal{H}$, we have that
\begin{align*}
T(U_0)\leq\liminf_{n\to \infty} T(U_{0,n}).
\end{align*}
Moreover if $U_{0,n}\to U_0$ and if $T<T(U_0)$, then $U_n\to U$ in $C([0,T],\mathcal{H})$, where $U_n$ and $U$ are the solutions of \eqref{duhamel} corresponding to the initial data $U_{0,n}$ and $U_0$.

\end{enumerate}

\end{proposition}
We prove Proposition \ref{lwp} in the same argument as \cite{CGOR}. To show the same as \cite{CGOR}, we only need to see the nonlinearity $F$ and the continuous semigroup $(e^{t\mathcal{A}})_{t\geq 0}$.

\begin{proof}
To apply the arguments as \cite[Section 2]{CGOR}, we estimate the nonlinearity $F$ and the continuous semigroup $(e^{t\mathcal{A}})_{t\geq 0}$.

First, we estimate $F$. we note that 
\begin{align}
|f(u)-f(v)|\lesssim (1+|u|^{p-1}+|v|^{p-1})|u-v|,
\end{align}
and so for any $R>0$, there is a constant $L(R)$ such that, for any $U,V\in \mathcal{H}$ with $\|U\|_{\mathcal{H}}<R$ and $\|V\|_{\mathcal{H}}<R$, we have 
\begin{align}\label{Fes}
\|F(U)-F(V)\|_{\mathcal{H}}\leq L(R)\|U-V\|_{\mathcal{H}}.
\end{align}

Second we estimate $(e^{t\mathcal{A}})_{t\geq 0}$. This has been estimated by \eqref{le}. Thus by \eqref{le} and \eqref{Fes} we complete the proof to apply \cite[Section 2]{CGOR}.

\end{proof}

\subsection{Uniform bound on global solutions} 
By subsection 2.2, we get the local well-posedness of \eqref{DNKG}. In this subsection, we get a uniform bound for global solutions of \eqref{DNKG}. Essentially, this proof is the same as \cite[Theorem 2.2]{CMY}. To prove the following theorem, we only need to note the effect of a delta potential.
\begin{theorem}\label{ge}
There exists a function $\mathcal{F}:[0,\infty)\to [0,\infty)$ such that $\mathcal{F}(0)=0$ and for any global solution $\vec{u}$ of \eqref{DNKG}, we have
\begin{align*}
\sup_{t\in [0,\infty)} \|\vec{u}(t)\|_{\mathcal{H}}\leq \mathcal{F}(\|\vec{u}(0)\|_{\mathcal{H}}).
\end{align*}
\end{theorem}

\begin{proof}
Let $\vec{u}$ be a global solution of \eqref{DNKG}. We define $E,M,W: [0,\infty)\to \mathbb{R}$ as
\begin{align*}
E(t)&=E_{\gamma}(\vec{u}(t)),
\\
M(t)&=\frac{1}{2}\|u(t)\|_{L^2}^2+\alpha \int_0^t\|u(s)\|_{L^2}^2ds,
\\
W(t)&=\frac{1}{2}\|\vec{u}(t)\|_{\mathcal{H}}^2-\frac{\gamma}{2}|u(t,0)|^2.
\end{align*}
By \eqref{le} we note that for all $(u,v)\in \mathcal{H}$
\begin{align}\label{douchi}
\|u\|_{H^1}^2+\|v\|_{L^2}^2 \sim \|u\|_{H^1}^2+\|v\|_{L^2}^2-\gamma|u(0)|^2.
\end{align} 
Then by direct computations, we have
\begin{align}
M^{\prime}(t)&=\int u(t){\partial}_tu(t)dx+\alpha\|u(t)\|_{L^2}^2 \label{Mp1}
\\
&=\int u(t){\partial}_tu(t)dx+2\alpha\int_0^t\int u(s){\partial}_tu(s)dxds+\alpha\|u(0)\|_{L^2}^2, \label{Mp2} 
\\
M^{\prime \prime}(t)&=\|{\partial}_tu(t)\|_{L^2}^2-\left(\|u(t)\|_{H^1}^2-\gamma|u(t,0)|^2\right)+\|u(t)\|_{L^{p+1}}^{p+1} \label{Mpp1}
\\
&=\frac{p+3}{2}\|{\partial}_tu(t)\|_{L^2}^2+\frac{p-1}{2}\left(\|u(t)\|_{H^1}^2-\gamma|u(t,0)|^2\right)-(p+1)E(t), \label{Mpp2}
\\
W^{\prime}(t)&=-2\alpha \|{\partial}_tu(t)\|_{L^2}^2+\int f(u(t)){\partial}_tu(t)dx. \label{W1}
\end{align}
Furthermore by \eqref{douchi}, \eqref{Mp1}, and the Cauchy-Schwarz inequality, there exists $c_1>0$ such that for all $t\geq 0$
\begin{align}\label{Mpes1}
|M^{\prime}(t)|\leq c_1 W(t).
\end{align}
Moreover by \eqref{energydecay} and \eqref{Mpp2}, there exists $c_2>0$ such that for all $t\geq 0$
\begin{align}\label{Mppes1}
M^{\prime \prime}(t)\geq c_2W(t)-(p+1)E(0).
\end{align}
In addition by \eqref{Mp2} and the Cauchy-Schwarz inequality,
\begin{align*}
|M^{\prime}(t)|&\leq \|u(t)\|_{L^2}\|{\partial}_tu(t)\|_{L^2}
\\
&\quad +2\alpha \left( \int_0^t \|u(s)\|_{L^2}^2ds\right)^{\frac{1}{2}}\left( \int_0^t\|{\partial}_tu(s)\|_{L^2}^2ds\right)^{\frac{1}{2}}
\\
&\quad +\alpha\|u(0)\|_{L^2}^2.
\end{align*}
Let $\epsilon>0$ to be chosen later, we estimate
\begin{align*}
|M^{\prime}|^2&\leq (1+\epsilon)\left[  \|u\|_{L^2}\|{\partial}_tu\|_{L^2}+2\alpha \left( \int_0^t \|u(s)\|_{L^2}^2ds\right)^{\frac{1}{2}}\left( \int_0^t\|{\partial}_tu(s)\|_{L^2}^2ds\right)^{\frac{1}{2}} \right]^2
\\
&\quad +\left(1+\frac{1}{\epsilon}\right)\alpha^2\|u(0)\|_{L^2}^4,
\end{align*}
because for $x>0, y>0,\epsilon>0$ the inequality $x+y\leq \sqrt{(a+\epsilon)x^2+(1+\frac{1}{\epsilon})y^2}$ holds. By using the Cauchy-Schwarz inequality again, we have
\begin{align}\label{Mode}
\begin{split}
|M^{\prime}(t)|^2&\leq (1+\epsilon)M(t)\left[ 2\|{\partial}_tu(t)\|_{L^2}^2+4\alpha\int_0^t\|{\partial}_tu(s)\|_{L^2}^2ds \right]
\\
&\quad +\left( 1+\frac{1}{\epsilon}\right)\alpha^2\|u(0)\|_{L^2}^4.
\end{split}
\end{align}
Furthermore by \eqref{energydecay} and \eqref{Mpp2}, 
\begin{align*}\label{Mpp3}
\begin{split}
M^{\prime \prime}(t)&=2\|{\partial}_tu(t)\|_{L^2}^2+(p-1)W(t)
\\
&\quad +2\alpha (p+1)\int_0^t\|{\partial}_tu(s)\|_{L^2}^2ds-(p+1)E(0).
\end{split}
\end{align*}
Now we fix $\epsilon$ such that 
\begin{align}
0<\epsilon<\left( \frac{p+7}{8}\right)^{\frac{1}{3}}-1 .
\end{align}
Then $M^{\prime \prime}$ satisfies
\begin{align}\label{Mppes2}
\begin{split}
M^{\prime \prime}(t)&\geq (1+\epsilon)^3\left[ 2\|{\partial}_tu(t)\|_{L^2}^2+4\alpha\int_0^t\|{\partial}_tu(s)\|_{L^2}^2ds \right]
\\
&\quad +\frac{p-1}{2}W(t)-(p+1)E(0).
\end{split}
\end{align}
We assume $\lim_{t\to \infty}M^{\prime}(t)=\infty$. Then by \eqref{Mode} and \eqref{Mppes2},
for all $t$ large enough, 
\begin{align}\label{Mode2}
(1+\epsilon)|M^{\prime}(t)|^2<M^{\prime \prime}(t)M(t). 
\end{align}
This inequality \eqref{Mode2} implies $\frac{d^2}{dt^2}\left[ M^{-\epsilon}(t)\right]<0$, and there exists $t_1>0$ such that  $\frac{d}{dt}\left[ M^{-\epsilon}(t_1)\right]<0$ and for all $t\geq t_1$,
\begin{align}
0\leq M^{-\epsilon}(t)\leq M^{-\epsilon}(t_1)+(t-t_1)\frac{d}{dt}\left[ M^{-\epsilon}(t_1)\right],
\end{align}
which is a contradiction. Therefore we obtain 
\begin{align}\label{Mpinf}
\liminf_{t\to \infty} M^{\prime}(t)<\infty.
\end{align}
Next we prove that there exists $C_3>0$ such that 
\begin{align}\label{Mpsup}
\sup_{t\in [0,\infty)}|M^{\prime}(t)|< C_3\left(|E(0)|+|M^{\prime}(0)|\right).
\end{align}
By \eqref{Mpes1} and \eqref{Mppes1}, we obtain 
\begin{align}
M^{\prime \prime}(t)\geq c_4|M^{\prime}(t)|-(p+1)E(0),
\end{align}
where $c_4=\frac{c_2}{c_1}$. Thus the following two inequalities hold :
\begin{align*}
\frac{d}{dt}\left[ c_4M^{\prime}(t)-(p+1)E(0)\right]\leq c_4\left[ c_4M^{\prime}(t)-(p+1)E(0)\right],
\\
\frac{d}{dt}\left[ -c_4M^{\prime}(t)+(p+1)E(0)\right]\leq -c_4\left[ -c_4M^{\prime}(t)+(p+1)E(0)\right].
\end{align*}
By \eqref{Mpinf}, we obtain \eqref{Mpsup}. In detail, see \cite[Theorem 2.2]{CMY}. 

\noindent Last, we prove the global bound
\begin{align}\label{Wb}
\sup_{t\in [0,\infty)} W(t)<\infty.
\end{align}
By integrating \eqref{Mppes1} on $(t,t+\tau)$,  where $\tau \in (0,1]$, there exists $C_5>0$ such that for all $t\geq 0$ and $0<\tau\leq 1$ 
\begin{align}\label{Wes2}
\int_t^{t+\tau}W(s)ds<C_5(|E(0)|+|M^{\prime}(0)|).
\end{align}
We note that $C_5$ is independent of $\tau$. By \eqref{W1}, there exists $C_6>0$ such that for all $t\geq0$
\begin{align*}
W^{\prime}(t)\leq -2\alpha\|{\partial}_tu(t)\|_{L^2}^2+\int|u(t)|^p|{\partial}_tu(t)|dx\leq C_6\|u(t)\|_{L^{2p}}^{2p}.
\end{align*}
For $\tau \in (0,1)$ and $t\geq \tau$, integrating on $(t-\tau ,t)$, we have
\begin{align*}
W(t)\leq W(t-\tau)+C_6\left(\int_{t-\tau}^t\|u(s)\|_{L^{2p}}^{2p}ds\right).
\end{align*}
Using \eqref{douchi}, \eqref{Wes2}, and the Sobolev inequality, there exists $C_7>0$ such that for all $0<\tau\leq 1$ and $t\geq 1$
\begin{align}\label{Wes3}
W(t)\leq W(t-\tau)+C_7\left( \int_{t-\tau}^tW(s)ds \right)^p.
\end{align}
If $0\leq t\leq 1$, by observing \eqref{Wes3} with $t=\tau$, we have
\begin{align}\label{Wes4}
W(t)\leq W(0)+C_7\left( \int_{0}^tW(s)ds \right)^p\leq W(0)+C_7\left\{C_5(|E(0)|+|M^{\prime}(0)|)\right\}^p.
\end{align} 
If $t\geq 1$, by integrating \eqref{Wes3} on $\tau\in (0,1)$, we have
\begin{align}\label{Wes5}
\begin{aligned}
W(t)&\leq \int_0^1W(t-\tau)d\tau+C_7\left\{C_5(|E(0)|+|M^{\prime}(0)|)\right\}^p
\\
&\leq C_5(|E(0)|+|M^{\prime}(0)|)+C_7\left\{C_5(|E(0)|+|M^{\prime}(0)|)\right\}^p.
\end{aligned}
\end{align}
Thus we obtain \eqref{Wb}. In particular, by \eqref{Wes4} and \eqref{Wes5} $W$ has an upper bound which depends only on $\|\vec{u}(0)\|_{\mathcal{H}}$, and by \eqref{douchi} we complete the proof.

\end{proof}

\section{Soliton resolution}

In this section,  we show Theorem \ref{sr}. Before proving Theorem \ref{sr}, we introduce the concentration-compactness argument.
\begin{proposition}\label{ssr}
Let $(u_n)_n\subset H^1$ be a bounded sequence satisfying
\begin{align}\label{wca}
\lim_{n\to \infty}\|-{\partial}_x^2u_n+u_n-\gamma{\delta}_0(x)u_n-|u_n|^{p-1}u_n\|_{H^{-1}}=0.
\end{align} 
Then there exist a subsequence $(v_n)_n\subset (u_n)_n$,  $K\geq 0$, signs $\sigma=0,\pm 1, \sigma_k=\pm 1$ for any $k\in \{1,2,\cdots,K\}$, and a sequence $(\xi_{k,n})_{k\in\{1,2,\cdots,K\}}\subset {\mathbb{R}}^K$ such that
\begin{align*}
\lim_{n\to \infty}& \|v_n-\sigma Q_{\gamma}-\sum_{k=1}^K{\sigma}_kQ(\cdot-\xi_{k,n})\|_{H^1}=0,
\\
\lim_{n\to \infty}&(\xi_{k+1,n}-\xi_{k,n})=\infty,\ \mbox{for}\ k=1,\cdots,K-1,
\\
\lim_{n\to \infty}&|\xi_{k,n}|=\infty\ \mbox{for}\ k=1,\cdots,K. 
\end{align*}
\end{proposition}

Concentration-compactness arguments are well-known. In fact, when $\gamma=0$, Proposition \ref{ssr} follows from \cite[Appendix A]{L} and \cite[Theorem~III.4]{L2}. 
To prove Proposition \ref{ssr}, we use the same arguments as them. We only need to observe the effect of a potential.

We denote $\langle \cdot, \cdot\rangle$ the $L^2$ scalar product.

\begin{proof}
Since $(u_n)_n$ is bounded in $H^1$, there exist $(\tilde{u}_n)_n\subset (u_n)_n$, $(\xi_n)_n\subset \mathbb{R}$, $\tilde{u}\in H^1(\mathbb{R})$, $c\geq 0$ and $\xi\in \{-\infty,0,\infty\}$ such that 
\begin{align*}
\lim_{n\to \infty}\|\tilde{u}_n\|_{L^{p+1}}^{p+1}=c,
\\
\tilde{u}_n(\cdot-\xi_n)\rightharpoonup \tilde{u}\ \mbox{in}\ H^1(\mathbb{R}),
\\
\lim_{n\to \infty}\xi_n=\xi .
\end{align*}
If $c=0$, $\tilde{u}_n$ converges to $0$ as $n\to \infty$. We assume $c>0$.

First we consider $\xi=0$. Then $\tilde{u}$ satisfies
\begin{align}\label{gode}
-{\partial}_x^2\tilde{u}+\tilde{u}-\gamma\delta_0\tilde{u}-|\tilde{u}|^{p-1}\tilde{u}=0.
\end{align}
Therefore if $|\gamma|\geq 2$ there does not exist a solution of \eqref{gode} and if $|\gamma|<2$  $\tilde{u}$ satisfies $\tilde{u}=\sigma Q_{\gamma}$, where $\sigma=\pm1$. We define $w_n=\tilde{u}_n-\sigma Q_{\gamma}$. Then we have as $n\to\infty$
\begin{align*}
w_n\rightharpoonup 0\ \mbox{in}\ H^1(\mathbb{R}).
\end{align*}
Furthermore by the Brezis–Lieb lemma, we have
\begin{align*}
\lim_{n\to \infty} \|w_n\|_{L^{p+1}}^{p+1}=c-\|Q_{\gamma}\|_{L^{p+1}}^{p+1}.
\end{align*} 
For every $\varphi\in H^1(\mathbb{R})$, we have
\begin{align*}
\langle -{\partial}_x^2w_n+w_n-\gamma \delta_0w_n-f(w_n),\varphi\rangle&=\langle -{\partial}_x^2\tilde{u}_n+\tilde{u}_n-\gamma \delta_0\tilde{u}_n-f(\tilde{u}_n),\varphi \rangle
\\
&\quad-\langle f(w_n)-f(\tilde{u}_n)+\sigma f(Q_\gamma),\varphi\rangle.
\end{align*}
Since $f(w_n)-f(\tilde{u}_n)+\sigma f(Q_\gamma)$ satisfies
\begin{align*}
|f(w_n)-f(\tilde{u}_n)+\sigma f(Q_\gamma)|&\lesssim e^{-|x|}(|w_n|^{p-1}+|w_n|),
\end{align*}
we have $\lim_{n\to \infty}\langle -{\partial}_x^2w_n+w_n-\gamma \delta_0w_n-f(w_n),\varphi\rangle=0$. Therefore we obtain 
\begin{align*}
\lim_{n\to \infty}\|-{\partial}_x^2w_n+w_n-\gamma \delta_0w_n-f(w_n)\|_{H^{-1}}=0.
\end{align*}
Second we consider $\xi=\pm \infty$. Then $\tilde{u}$ satisfies 
\begin{align*}
-{\partial}_x^2\tilde{u}+\tilde{u}-f(\tilde{u})=0.
\end{align*}
Therefore we may assume that $\tilde{u}=\sigma Q$, where $\sigma=\pm1$. Then we define $Q_n=\sigma Q(\cdot+\xi_n)$ and $w^{\prime}_n=\tilde{u}_n-Q_n$. Then by the Brezis–Lieb lemma, we have
\begin{align*}
\lim_{n\to \infty}\|w^{\prime}_n\|_{L^{p+1}}^{p+1}=c-\|Q\|_{L^{p+1}}^{p+1}.
\end{align*}
Now we fix $\varphi \in H^1(\mathbb{R})$. Then we have
\begin{align*}
\langle -{\partial}_x^2w^{\prime}_n+w^{\prime}_n-\gamma \delta_0w^{\prime}_n-f(w^{\prime}_n),\varphi\rangle&=\langle -{\partial}_x^2\tilde{u}_n+\tilde{u}_n-\gamma \delta_0\tilde{u}_n-f(\tilde{u}_n),\varphi\rangle
\\
&\quad+\langle \gamma\delta_0Q_n,\varphi\rangle
\\
&\quad-\langle f(w^{\prime}_n)-f(\tilde{u}_n)+f(Q_n),\varphi\rangle. 
\end{align*}
By direct computation, we have
\begin{align*}
| f(w^{\prime}_n)-f(\tilde{u}_n)+f(Q_n)|&\lesssim |Q_n||\tilde{u}_n|^{p-1}+|Q_n|^{p-1}|\tilde{u}_n|.
\end{align*}
Therefore for any $R>0$, we obtain 
\begin{align*}
|\langle f(w^{\prime}_n)-f(\tilde{u}_n)+f(Q_n),\varphi\rangle|\lesssim \|\varphi\|_{L^2(|x|>R)}+e^{-|x_n|}e^R\|\varphi\|_{H^1},
\end{align*}
which implies $\lim_{n\to \infty}\langle -{\partial}_x^2w^{\prime}_n+w^{\prime}_n-\gamma \delta_0w^{\prime}_n-f(w^{\prime}_n),\varphi\rangle=0$. Thus we obtain
\begin{align*}
\lim_{n\to \infty}\|-{\partial}_x^2w^{\prime}_n+w^{\prime}_n-\gamma \delta_0w^{\prime}_n-f(w^{\prime}_n)\|_{H^{-1}}=0.
\end{align*}
Gathering these arguments, there exist $\mathcal{Q}\in \{\pm Q_{\gamma},\pm Q\}$ and $(\xi_n)_n\subset \mathbb{R}$ such that 
\begin{align*}
&\lim_{n\to \infty}\| \tilde{w}_n\|_{L^{p+1}}^{p+1}\leq \lim_{n\to \infty}\|\tilde{u}_n\|_{L^{p+1}}^{p+1}-\min{(\|Q\|_{L^{p+1}}^{p+1},\|Q_{\gamma}\|_{L^{p+1}}^{p+1})},
\\
&\lim_{n\to \infty}\|-{\partial}_x^2\tilde{w}_n+\tilde{w}_n-\gamma \delta_0\tilde{w}_n-f(\tilde{w}_n)\|_{H^{-1}}=0,
\end{align*}
where $\tilde{w}_n=\tilde{u}_n-\mathcal{Q}(\cdot+\xi_n)$. We note that if $|\gamma|\geq 2$, $\min{(\|Q\|_{L^{p+1}}^{p+1},\|Q_{\gamma}\|_{L^{p+1}}^{p+1})}=\|Q\|_{L^{p+1}}^{p+1}$. Therefore by repeatedly applying this argument, we complete the proof.

\end{proof}

Now we prove Theorem \ref{sr}. First we define 
\begin{align*}
E_{0,R}&=\mathbb{R},
\\
E_{K,R}&=\{ z\in{\mathbb{R}}^K: |z_k|>R\ \mbox{for}\ 1\leq k\leq K\ \mbox{and}\ z_{k+1}-z_k>R\ \mbox{for} \ 1\leq k\leq K-1\}, 
\end{align*}
where $K\geq 1$ and $R>0$. Furthermore we define
\begin{align*}
d_{K,R,\varsigma,\sigma}(u)&=\inf_{z\in E_{K,R}}\|u-\varsigma Q_{\gamma}-\sum_{k=1}^K{\sigma}_kQ(\cdot-z_k)\|_{H^1},
\\
\mathcal{N}_{K,R,\varsigma,\sigma,\epsilon}&=\{u\in H^1(\mathbb{R}): d_{K,R,\varsigma,\sigma}(u)<\epsilon \},
\end{align*}
where 
\begin{align*}
\varsigma \in \{0,\pm1\} ,\ \sigma=(\sigma_k)_{1\leq k\leq K} \in \{\pm1\}^K,\ K\geq 0,\ R>0,\ \epsilon>0.
\end{align*}

We note that when $R>0$ is large enough and $\epsilon>0$ is small enough, the family of sets $(\mathcal{N}_{K,R,\varsigma,\sigma,\epsilon})$ is disjoint.

\begin{proof}[{Proof of Theorem \ref{sr}}]
Let $\vec{u}$ be a global solution of \eqref{DNKG}. First, we prove 
\begin{align}
&\lim_{t\to \infty}\|{\partial}_tu(t)\|_{L^2}=0, \label{ptc}
\\
&\lim_{t\to \infty}\|-{\partial}_x^2u(t)+u(t)-\gamma{\delta}_0u(t)-|u(t)|^{p-1}u(t)\|_{H^{-1}}=0. \label{uh-1}
\end{align} 
Let
\begin{align*}
\vec{v}(t)=
\begin{pmatrix}
v(t)
\\
{\partial}_tv(t)
\end{pmatrix}
=
\begin{pmatrix}
{\partial}_tu(t)
\\
{\partial}_t^2u(t)
\end{pmatrix}
.
\end{align*}
By Proposition \ref{lineares}, we obtain
\begin{align}\label{duhameles}
\vec{v}(t)=e^{t\mathcal{A}}\vec{v}(0)+\int_0^t e^{(t-s)\mathcal{A}}
\begin{pmatrix}
0
\\
|u(s)|^{p-1}{\partial}_tu(s)
\end{pmatrix}
ds,
\end{align}
and for all $t\geq 0$,
\begin{align*}
\| e^{t\mathcal{A}}\vec{v}(0)\|_{L^2\times H^{-1}}&\lesssim e^{-\kappa t}.
\end{align*}
By \eqref{energydecay}, $v={\partial}_tu\in L^2([0,\infty)\times \mathbb{R})$, and by Theorem \ref{ge}, $u\in L^{\infty}([0,\infty),H^1)$. Therefore we have 
\begin{align*}
\| \int_0^{\frac{t}{2}}e^{(t-s)\mathcal{A}}\vec{w}(s)ds\|_{L^2\times H^{-1}}&\lesssim \int_0^{\frac{t}{2}} e^{-\kappa(t-s)}\|u\|_{L^{\infty}([0,\infty),H^1)}^{p-1}\|v(s)\|_{L^2}ds
\\
&\lesssim e^{-\frac{\kappa t}{2}} \|u\|_{L^{\infty}([0,\infty),H^1)}^{p-1}\|v\|_{L^2([0,\infty)\times \mathbb{R})},
\\
\| \int_{\frac{t}{2}}^{t}e^{(t-s)\mathcal{A}}\vec{w}(s)ds\|_{L^2\times H^{-1}}&\lesssim \int_{\frac{t}{2}}^t e^{-\kappa(t-s)}\|u\|_{L^{\infty}([0,\infty),H^1)}^{p-1}\|v(s)\|_{L^2}ds
\\
&\lesssim \|u\|_{L^{\infty}([0,\infty),H^1)}^{p-1}\|v\|_{L^2((\frac{t}{2},\infty)\times \mathbb{R})},
\end{align*}
where
\begin{align*}
\vec{w}(t)=
\begin{pmatrix}
0
\\
|u(t)|^{p-1}{\partial}_tu(t)
\end{pmatrix}
.
\end{align*}
Thus we obtain \eqref{ptc} and \eqref{uh-1}.

Then, $\left(u(n)\right)_n\subset H^1(\mathbb{R})$ is bounded and satisfies \eqref{wca}. Therefore by Proposition \ref{ssr}, there exist $\tilde{K}\geq0$, $(\tilde{t}_n)_n$,\ $\tilde{\varsigma}\in \{0,\pm1\}$, $\tilde{\sigma}=(\tilde{\sigma}_k)_{1\leq k\leq \tilde{K}}\in \{\pm1\}^{\tilde{K}}$ such that for any $R>0$,
\begin{align*}
\lim_{n\to \infty} d_{\tilde{K},R,\tilde{\varsigma},\tilde{\sigma}}(u(\tilde{t}_n))=0.
\end{align*}

Next we prove that for any sequence $t_n\to\infty$, $u$ satisfies 
\begin{align}\label{3uconv}
\lim_{n\to \infty} d_{\tilde{K},R,\tilde{\varsigma},\tilde{\sigma}}(u(t_n))=0.
\end{align}
We assume that there exist a sequence $t_n\to \infty$ and $R>0$  such that 
\begin{align*}
\limsup_{n\to \infty} d_{\tilde{K},R,\tilde{\varsigma},\tilde{\sigma}}(u(t_n))>0.
\end{align*}
Then there exist a subsequence $(t_n^{\prime})_n\subset (t_n)_n$ and sufficiently small $\epsilon>0$ such that for all $n\in\mathbb{N}$
\begin{align}\label{unotin}
u(t_n^{\prime})\notin \mathcal{N}_{\tilde{K},R,\tilde{\varsigma},\tilde{\sigma},\epsilon}.
\end{align}
Since $(u(t_n^{\prime}))_n$ is bounded and satisfies \eqref{wca}, by Proposition \ref{ssr} there exist a subsequence $(t_n^{\prime \prime})_n\subset (t_n^{\prime})_n$, $\hat{K},\ \hat{\varsigma}$, and $\hat{\sigma}$ such that 
\begin{align*}
\lim_{n\to\infty}d_{\hat{K},R,\hat{\varsigma},\hat{\sigma}}(u(t_n^{\prime \prime}))=0.
\end{align*}
We note that by \eqref{unotin}, $(\hat{K},\hat{\varsigma},\hat{\sigma})\neq (\tilde{K},\tilde{\varsigma},\tilde{\sigma})$ holds. Therefore when $n,m\in\mathbb{N}$ satisfy $u(t_n)\in \mathcal{N}_{\tilde{K},R,\tilde{\varsigma},\tilde{\sigma},\epsilon}$ and $u(t_m^{\prime})\in \mathcal{N}_{\hat{K},R,\hat{\varsigma},\hat{\sigma},\epsilon}$, there exists $s\in\mathbb{R}$ between $t_n$ and $t_m^{\prime}$ such that 
\begin{align*}
u(s)\notin\bigcup_{K\geq0,\varsigma\in\{0,\pm1\},\sigma\in \{\pm1\}^K} \mathcal{N}_{K,R,\varsigma,\sigma,\epsilon}.
\end{align*}
By this argument, there exists a sequence $s_n\to\infty$ such that for all $n\in\mathbb{N}$,
\begin{align}\label{unotin2}
u(s_n)\notin \bigcup_{K\geq0,\varsigma\in\{0,\pm1\},\sigma\in \{\pm1\}^K} \mathcal{N}_{K,R,\varsigma,\sigma,\epsilon}.
\end{align}
Then, $(u(s_n))_n$ is bounded and $u(s_n)$ satisfies \eqref{wca}. Therefore by Proposition \ref{ssr}, there exist a subsequence $(s_n^{\prime})_n\subset (s_n)_n$, $\dot{K}, \dot{\varsigma},\dot{\sigma}$ such that
\begin{align*}
\lim_{n\to\infty}d_{\dot{K},R,\dot{\varsigma},\dot{\sigma}}(u(s_n^{\prime}))=0,
\end{align*}
which contradicts \eqref{unotin2}. Therefore  for any $t_n\to\infty$, $u$ satisfies \eqref{3uconv}, and therefore we have
\begin{align*}
\lim_{t\to\infty}d_{\tilde{K},R,\tilde{\varsigma},\tilde{\sigma}}(u(t))=0.
\end{align*}
We prove Theorem \ref{sr}.  We define $z(t)=(z_k(t))\in\mathbb{R}^{\tilde{K}}$ as 
\begin{align*}
\|u(t)-\tilde{\varsigma}Q_{\gamma}-\sum_{k=1}^{\tilde{K}}\tilde{\sigma}_k Q(\cdot-z_k(t))\|_{H^1}=\inf_{z\in\mathbb{R}^{\tilde{K}}}\|u(t)-\tilde{\varsigma} Q_{\gamma}-\sum_{k=1}^{\tilde{K}}{\tilde{\sigma}}_kQ(\cdot-z_k)\|_{H^1}.
\end{align*}
We note that for any $R>0$, the minimum is attained for $z(t)$ satisfying $z(t)\in E_{R,\tilde{K}}$ when $t$ is large enough. Therefore we obtain \eqref{sr-1} with $K=\tilde{K},\ \sigma=\tilde{\varsigma}$, and $\sigma_k=\tilde{\sigma}_k$ for $1\leq k\leq \tilde{K}$, and we complete the proof.

\end{proof}

\section{Global dynamics below the ground state energy}

\subsection{Minimizing problems and variational structure}

First we collect properties of the ground state energy. First we define $K_{\gamma}, J_{\gamma}:H^1(\mathbb{R})\to \mathbb{R}$ as 
\begin{align*}
K_{\gamma}(\varphi)&=\| \varphi\|_{H^1}^2-\gamma|\varphi (0)|^2-\|\varphi \|_{L^{p+1}}^{p+1} ,
\\
J_{\gamma}(\varphi)&=\frac{1}{2}\left(\|\varphi \|_{H^1}^2-\gamma|\varphi(0)|^2\right)-\frac{1}{p+1}\|\varphi\|_{L^{p+1}}^{p+1},
\end{align*}
and we consider the following minimizing problems.
\begin{align*}
n_{\gamma}&=\inf \{ J_{\gamma}(\varphi): \varphi\in H^1(\mathbb{R})\setminus \{0\} ,  K_{\gamma}(\varphi)=0\},
\\
r_{\gamma}&=\inf \{J_{\gamma}(\varphi):\varphi \in H_{rad}^1(\mathbb{R})\setminus \{0\} , K_{\gamma}(\varphi)=0\},
\end{align*}
where $H_{rad}^1(\mathbb{R}):=\{ \varphi \in H^1(\mathbb{R}): \varphi (x)=\varphi (-x)\}$. In \cite{FJ}, we  get the following statements about $n_{\gamma}$ and $r_{\gamma}$.
\begin{proposition}\label{vs0}
The following properties hold.
\begin{enumerate}
\item if $0\leq\gamma<2$, $Q_{\gamma}$ is a minimizer of $n_{\gamma}$ and  $r_{\gamma}$.

\item if $-2<\gamma<0$, $Q_{\gamma}$ is a minimizer of $r_{\gamma}$. Furthermore $n_{\gamma}=J_0(Q)$ has no minimizer.

\item if $\gamma\leq -2$, $n_{\gamma}=J_0(Q)$ has no minimizer and $r_{\gamma}=2J_0(Q)$ has no minimizer.
\end{enumerate}
\end{proposition}

Furthermore we get a uniform bound on $K_{\gamma}$.

\begin{proposition}\label{vs1}
\begin{enumerate}
\item For any $\varphi\in H^1(\mathbb{R})$ with $J_{\gamma}(\varphi)<n_{\gamma}$, we have 
\begin{align}\label{vses}
K_{\gamma}(\varphi)\gtrsim \min{\{n_{\gamma}-J_{\gamma}(\varphi), \|\varphi\|_{H^1}^2\}} \ or \ K_{\gamma}(\varphi)\lesssim -(n_{\gamma}-J_{\gamma}(\varphi)). 
\end{align}
\item For any $\varphi\in H_{rad}^1(\mathbb{R})$ with $J_{\gamma}(\varphi)<r_{\gamma}$, we have 
\begin{align}\label{vses2}
K_{\gamma}(\varphi)\gtrsim \min{\{r_{\gamma}-J_{\gamma}(\varphi), \|\varphi\|_{H^1}^2\}} \ or \ K_{\gamma}(\varphi)\lesssim -(r_{\gamma}-J_{\gamma}(\varphi)). 
\end{align}

\end{enumerate}
\end{proposition}

This uniform bound is useful to classify the global behavior of solutions whose energy is less than ground state energy for the nonlinear dispersive equations. To prove Proposition \ref{vs1}, we refer to \cite[Proposition 2.18]{II} and \cite[Lemma 2.12]{IMN}.

\begin{proof}
First we prove (1). We may assume $\varphi \neq 0$. We define $s(\lambda)=J_{\gamma}(e^{\lambda}\varphi)$. Then we have
\begin{align}
s(\lambda)&=\frac{e^{2\lambda}}{2}\left(\|\varphi\|_{H^1}^2-\gamma|\varphi(0)|^2\right)-\frac{e^{(p+1)\lambda}}{p+1}\|\varphi\|_{L^{p+1}}^{p+1}, \label{ses}
\\
s^{\prime}(\lambda)&=e^{2\lambda}\left(\|\varphi\|_{H^1}^2-\gamma|\varphi(0)|^2\right)-e^{(p+1)\lambda}\|\varphi\|_{L^{p+1}}^{p+1}, \label{spes}
\\ 
s^{\prime \prime}(\lambda)&=2e^{2\lambda}\left(\|\varphi\|_{H^1}^2-\gamma|\varphi(0)|^2\right)-(p+1)e^{(p+1)\lambda}\|\varphi\|_{L^{p+1}}^{p+1}. \label{sppes}
\end{align}
By \eqref{spes} and \eqref{sppes}, we obtain $s^{\prime}(0)=K_{\gamma}(\varphi)$ and
\begin{align}\label{spsppes}
 s^{\prime \prime}(\lambda)\leq 2s^{\prime}(\lambda).
\end{align}
First we consider $K_{\gamma}(\varphi)<0$. Then we have $s^{\prime}(0)<0$ and $s^{\prime}(\lambda)>0$ for sufficiently small $\lambda<0$. Therefore there exists $\lambda_0<0$ such that $s^{\prime}(\lambda_0)=0$. By integrating \eqref{spsppes} on $[\lambda_0,0]$, we have
\begin{align}\label{sspes}
s^{\prime}(0)-s^{\prime}(\lambda_0)\leq 2(s(0)-s(\lambda_0)).
\end{align}
By \eqref{sspes} and $s^{\prime}(\lambda_0)=0$,  we obtain 
\begin{align*}
K_{\gamma}(\varphi)=s^{\prime}(0)\leq -2(n_{\gamma}-J_{\gamma}(\varphi)).
\end{align*}
Second we consider $K_{\gamma}(\varphi)\geq 0$. Since $\varphi \neq 0$, we consider $K_{\gamma}(\varphi)>0$. If 
\begin{align*}
K_{\gamma}(\varphi)\geq \frac{1}{3(p-1)}\|\varphi\|_{L^{p+1}}^{p+1},
\end{align*}
we have
\begin{align*}
(1+\frac{1}{3(p-1)})K_{\gamma}(\varphi)\geq \frac{1}{3(p-1)}\left(\|\varphi\|_{H^1}^2-\gamma|\varphi(0)|^2\right).
\end{align*}
By \eqref{deqes}, we obtain $K_{\gamma}(\varphi)\gtrsim \|\varphi\|_{H^1}^2$. If 
\begin{align*}
K_{\gamma}(\varphi)<\frac{1}{3(p-1)}\|\varphi\|_{L^{p+1}}^{p+1},
\end{align*}
there exists $\lambda_1>0$ such that 
\begin{align}\label{spes1}
0<s^{\prime}(\lambda)<\frac{1}{3(p-1)}e^{(p+1)\lambda}\|\varphi\|_{L^{p+1}}^{p+1}
\end{align}
for $0<\lambda<\lambda_1$. Then by direct calculation, we have
\begin{align*}
s^{\prime\prime}(\lambda)&=2s^{\prime}(\lambda)-(p-1)e^{(p+1)\lambda}\|\varphi\|_{L^{p+1}}^{p+1}
\\
&<\{2-3(p-1)^2\}s^{\prime}(\lambda)
\\
&<-s^{\prime}(\lambda)
\end{align*}
for $0<\lambda<\lambda_1$. In particular, as long as \eqref{spes1} holds,  $s^{\prime}$ decreases. Furthermore $\lim_{\lambda\to\infty}s^{\prime}(\lambda)=-\infty$ holds and we may assume $s^{\prime}(\lambda_1)=0$. Then by integrating $s^{\prime \prime}<-s^{\prime}$ on $[0,\lambda_1]$, we have
\begin{align*}
s^{\prime}(\lambda_1)-s^{\prime}(0)<-(s(\lambda_1)-s(0)).
\end{align*}
Thus we obtain
\begin{align*}
K_{\gamma}(\varphi)>n_{\gamma}-J_{\gamma}(\varphi).
\end{align*}
Gathering these arguments, we obtain \eqref{vses}.

We note that we only use the variational structure to obtain \eqref{vses}. Therefore we obtain \eqref{vses2} by the same argument.

\end{proof}

\subsection{Classification of global dynamics}
In this subsection, we classify the global behavior of solutions whose energy is less than the ground state energy. To classify the global behavior of solutions, we use the argument as \cite{PS}.
\begin{proposition}\label{classify}
Let $\vec{u}$ be a solution of \eqref{DNKG}.
\begin{enumerate}
\item  If $\vec{u}$ satisfies $E_{\gamma}(\vec{u}(0))<n_{\gamma}$ and $K_{\gamma}(u(0))\geq0$, then $\vec{u}$ is global and converges exponentially to $0$ in $\mathcal{H}$ as $t\to\infty$.
\item If $\vec{u}$ satisfies $E_{\gamma}(\vec{u}(0))<n_{\gamma}$ and $K_{\gamma}(u(0))<0$, then $\vec{u}$ blows up in finite time.

\item If $\vec{u}$ is even and satisfies $E_{\gamma}(\vec{u}(0))<r_{\gamma}$ and $K_{\gamma}(u(0))\geq0$, then $\vec{u}$ is global and converges exponentially to $0$ in $\mathcal{H}$ as $t\to\infty$.

\item If $\vec{u}$ is even and satisfies $E_{\gamma}(\vec{u}(0))<r_{\gamma}$ and $K_{\gamma}(u(0))<0$, then $\vec{u}$ blows up in finite time.

\end{enumerate}
\end{proposition}
Before we prove Proposition \ref{classify}, we define $\mathcal{P}:\mathcal{H}\to \mathbb{R}$ as 
\begin{align*}
\mathcal{P}(\varphi_1,\varphi_2)=\int \varphi_1\varphi_2+\alpha\|\varphi_1\|_{L^2}^2 .
\end{align*}

\begin{proof}
First we prove (1). Since $J_{\gamma}(u(0))<n_{\gamma}$, for all $t\geq 0$ $K_{\gamma}(u(t))\geq 0$ holds. Then we have for $t\geq 0$
\begin{align*}
 E_{\gamma}(\vec{u})&=\frac{1}{p+1}K_{\gamma}(u)+\frac{1}{2}\|{\partial}_tu\|_{L^2}^2+\frac{p-1}{2(p+1)}\left(\|u\|_{H^1}^2-\gamma|u(0)|^2\right)
 \\
 &\geq \frac{1}{2}\|{\partial}_tu\|_{L^2}^2+\frac{p-1}{2(p+1)}\left(\|u\|_{H^1}^2-\gamma|u(0)|^2\right),
\end{align*}
and therefore we obtain 
\begin{align}\label{simeq}
E_{\gamma}(\vec{u}(t))\sim \| \vec{u}(t)\|_{\mathcal{H}}^2
\end{align}
for all $t\geq 0$. By direct calculation and \eqref{simeq} and $\mathcal{P}(\vec{u})\lesssim \|\vec{u}\|_{\mathcal{H}}$, there exists $C>0$ such that 
\begin{align*}
{\partial}_t[\mathcal{P}(\vec{u})+CE_{\gamma}(\vec{u})]&=(1-2\alpha C)\|{\partial}_tu\|_{L^2}^2-K_{\gamma}(u)\leq -\frac{1}{C}[\mathcal{P}(\vec{u})+CE_{\gamma}(\vec{u})],
\\
\|\vec{u}(t)\|_{\mathcal{H}}&\sim \mathcal{P}(\vec{u}(t))+CE_{\gamma}(\vec{u}(t)).
\end{align*}
Thus we obtain
\begin{align*}
\|\vec{u}(t)\|_{\mathcal{H}}\lesssim e^{-\frac{t}{C}}\|\vec{u}(0)\|_{\mathcal{H}},
\end{align*}
and so the desired estimate.

Next we prove (2). We assume that $\vec{u}$ is global. By Proposition \ref{vs1} and \eqref{energydecay}, there exists $c>0$ such that 
\begin{align*}
{\partial}_t\mathcal{P}(\vec{u})&=\|{\partial}_tu\|_{L^2}^2-K_{\gamma}(u)
\\
&\geq c(n_{\gamma}-J_{\gamma}(u))
\\
&\geq c(n_{\gamma}-E_{\gamma}(\vec{u}))
\\
&\geq c\left(n_{\gamma}-E_{\gamma}(\vec{u}(0))\right).
\end{align*}
By the above estimate, $\lim_{t\to \infty}\mathcal{P}(\vec{u}(t))=\infty$ holds and therefore $\lim_{t\to \infty} \|\vec{u}(t)\|_{\mathcal{H}}=\infty$ holds, which contradicts Theorem \ref{ge}. Thus we have proved (2).

Last, we prove (3) and (4). Let $\vec{u}$ be an even. Then \eqref{vses2} holds, and we can prove them by the same arguments as (1) and (2). Therefore we complete the proof.

\end{proof}

\begin{remark}
We assume that a solution $\vec{u}$ of \eqref{DNKG} satisfies $\lim_{t\to\infty}\|\vec{u}(t)\|_{\mathcal{H}}=0$. Then $\vec{u}$ satisfies $\|\vec{u}(t)\|_{\mathcal{H}}\ll1$ for some $t\in\mathbb{R}$, and therefore $\vec{u}(t)$ satisfies $E_{\gamma}(\vec{u}(t))<n_{\gamma}$ and $K_{\gamma}(u(t))\geq 0$. Therefore By (1) of Proposition \ref{classify}, $\vec{u}$ converges exponentially to $0$ as $t\to\infty$. By this argument, When a global solution $\vec{u}$ satisfies \eqref{sr-1} with $\sigma=0$ and $K=0$, $\vec{u}$ converges exponentially to $0$ as $t\to\infty$.

\end{remark}

We define $\mathcal{H}_{rad}$ as 
\begin{align*}
\mathcal{H}_{rad}=\{(\varphi_1,\varphi_2)\in \mathcal{H}:\varphi_1(x)=\varphi_1(-x),\ \varphi_2(x)=\varphi_2(-x)\}.
\end{align*}
Furthermore we define $\mathcal{B},\ \mathcal{B}_{rad},\ \mathcal{V},\ \mathcal{V}_{rad}$ by
\begin{align*}
\mathcal{B}&=\{ \vec{\varphi}\in \mathcal{H}: \vec{u}\ \mbox{blows up in finite time}\},
\\
\mathcal{B}_{rad}&=\{ \vec{\varphi}\in \mathcal{H}_{rad}: \vec{u}\ \mbox{blows up in finite time}\},
\\
\mathcal{V}&=\{ \vec{\varphi}\in \mathcal{H}: \vec{u}\ \mbox{is global and} \ \lim_{t\to \infty}\|\vec{u}(t)\|_{\mathcal{H}}=0\},
\\
\mathcal{V}_{rad}&=\{ \vec{\varphi}\in \mathcal{H}_{rad}: \vec{u}\ \mbox{is global and} \ \lim_{t\to \infty}\|\vec{u}(t)\|_{\mathcal{H}}=0\},
\end{align*}
where $\vec{u}$ is the solution of \eqref{DNKG} with initial data $\vec{\varphi}$. Then by Proposition \ref{classify}, $\mathcal{V}$ and $\mathcal{V}_{rad}$ are open. On the other hand, by Proposition \ref{lwp} and Theorem \ref{ge}, we obtain that the set of global solutions of \eqref{DNKG} is closed. Therefore $\mathcal{B}$ and $\mathcal{B}_{rad}$ are open.

\subsection{Proof of Theorem \ref{existencesoliton}}

In this subsection, we prove Theorem \ref{existencesoliton}. Let $z>0$ be large enough.  We define 
\begin{align*}
Q_{\varsigma}&=Q(\cdot-z)+\varsigma Q(\cdot+z),
\\
Q_{\lambda,\varsigma}&=e^{\lambda}Q_{\varsigma},
\\
\vec{Q}_{\lambda,\varsigma}&=(Q_{\lambda,\varsigma},0),
\\
\mathcal{A}_{\varsigma}&=\{ \vec{Q}_{\lambda,\varsigma}: -1\leq \lambda\leq 1\},
\end{align*}
where $\varsigma=0,1$ and $-1\leq \lambda\leq 1$. Furthermore we define 
\begin{align*}
x_{\varsigma}(\lambda)=\frac{e^{2\lambda}}{2}\|Q_{\varsigma} \|_{H^1}^2-\frac{e^{(p+1)\lambda}}{p+1}\|Q_{\varsigma}\|_{L^{p+1}}^{p+1}-\frac{\gamma e^{2\lambda}}{2}|Q_{\varsigma}(0)|^2.
\end{align*}
Then by direct calculation, we have
\begin{align*}
x_{\varsigma}^{\prime}(\lambda)&=e^{2\lambda}\|Q_{\varsigma}\|_{H^1}^2-e^{(p+1)\lambda}\|Q_{\varsigma}\|_{L^{p+1}}^{p+1}-\gamma e^{2\lambda}|Q_{\varsigma}(0)|^2,
\\
x_{\varsigma}^{\prime \prime}(\lambda)&=2e^{2\lambda}\|Q_{\varsigma}\|_{H^1}^2-(p+1)e^{(p+1)\lambda}\|Q_{\varsigma}\|_{L^{p+1}}^{p+1}-2\gamma e^{2\lambda}|Q_{\varsigma}(0)|^2.
\end{align*}
Now we start to prove Theorem \ref{existencesoliton}.
\begin{proof}[{Proof of Theorem \ref{existencesoliton}}]
First we prove (1) of Theorem \ref{existencesoliton}. We consider $\varsigma=0$. Since $K_0(Q)=0$ holds, we have
\begin{align*}
x_0^{\prime}(0)&=-\gamma|Q(z)|^2,
\\
x_0^{\prime \prime}(0)&=-(p-1)\|Q\|_{L^{p+1}}^{p+1}-2\gamma|Q(z)|^2.
\end{align*}
We recall that $z$ is large enough. Thus there exists  $\lambda^{\prime}<0$ close to $0$ such that  $x_0^{\prime}(\lambda^{\prime})=0$. Moreover there exists $C>0$ and $0<\delta<1$ such that for $-\delta\leq \lambda\leq\delta$  
\begin{align*}
|x_0^{\prime}(\lambda)-x_0^{\prime}(\lambda^{\prime})|\geq C|\lambda-\lambda^{\prime}|.
\end{align*}
We note that $\delta$ is independent of $z$ if $z$ is large enough. Then  for $-\delta\leq \lambda<\lambda^{\prime}$, $x_0^{\prime}(\lambda)>0$ holds and for $\lambda^{\prime}<\lambda\leq \delta$, $x_0^{\prime}(\lambda)<0$ holds. Therefore for $-\delta\leq \lambda\leq \delta$, we have
\begin{align*}
x_0(\lambda^{\prime})-x_0(\lambda)&=\int_{\lambda}^{\lambda^{\prime}}x_0^{\prime}(s)-x_0^{\prime}(\lambda^{\prime})ds
\\
&= \left|\int_{\lambda^{\prime}}^{\lambda}|x_0^{\prime}(s)-x_0^{\prime}(\lambda^{\prime})|ds\right|
\\
&\geq C\left|\int_{\lambda^{\prime}}^{\lambda}|s-\lambda^{\prime}|ds\right|
\\
&=\frac{C}{2}|\lambda-\lambda^{\prime}|^2.
\end{align*}
Thus we have
\begin{align*}
x_0^{\prime}(-\delta)>0 ,\ x_0^{\prime}(\delta)<0,
\\
x_0(-\delta)<n_{\gamma} ,\ x_0(\delta)<n_{\gamma},
\end{align*}
because $\lambda^{\prime}$ is close to $0$. Therefore by Proposition \ref{classify}, $\vec{Q}_{-\delta,0}\in \mathcal{V}$ and $\vec{Q}_{\delta,0}\in \mathcal{B}$.
Since $\mathcal{A}$ is connected and $\mathcal{B}$ and $\mathcal{V}$ are open, there exists $-1\leq \lambda^{\star}\leq 1$ such that $\vec{Q}_{\lambda^{\star},0}\notin \mathcal{B}$ and  $\vec{Q}_{\lambda^{\star},0}\notin \mathcal{V}$. Moreover we have
\begin{align}\label{soliene}
E_{\gamma}(\vec{Q}_{\lambda^{\star},0})<E_{\gamma}(\vec{Q}_{\lambda^{\prime},0}).
\end{align}
We note that $E_{\gamma}(\vec{Q}_{\lambda^{\prime},0})$ is close to $n_{\gamma}$. Thus by Theorem \ref{sr} and \eqref{soliene}, the solution $\vec{u}_{\star,0}$ of \eqref{DNKG} with initial data $\vec{Q}_{\lambda^{\star},0}$ satisfies 
\begin{align*}
\lim_{t\to \infty}\left(\| u_{\star,0}(t)-\sigma^{\prime} Q(\cdot-z(t))\|_{H^1}+\|{\partial}_tu_{\star,0}(t)\|_{L^2}\right)=0,
\end{align*}
where $\sigma^{\prime}=\pm 1$ and $\lim_{t\to \infty}|z(t)|=\infty$. This completes the proof of (1) of Theorem \ref{existencesoliton}, since \eqref{invariance} holds.

Next we prove (2) of Theorem \ref{existencesoliton}. We consider $\varsigma=1$.  Then $\mathcal{A}_1\subset \mathcal{H}_{rad}$ holds. Furthermore  by Proposition \ref{vs0} and  direct computation, we have
\begin{align*}
|x_1(0)-r_{\gamma}|&\lesssim e^{-z},
\\
|x_1^{\prime}(0)|&\lesssim e^{-z},
\\
\left|x_1^{\prime \prime}(0)+2(p-1)\|Q\|_{L^{p+1}}^{p+1}\right|&\lesssim e^{-z}. 
\end{align*}
Thus by the same argument as (1) of Theorem \ref{existencesoliton}, we complete the proof of (2) of Theorem \ref{existencesoliton}.

\end{proof}

\section{Dynamics close to solitary waves}

\subsection{Basic properties of the ground state }

In this subsection, we collect some basic properties. First, the linearized operator $\mathcal{L}$ around $Q$ writes
\begin{align}\label{Ldef}
\mathcal{L}=-{\partial}_x^2+1-pQ^{p-1}.
\end{align}
We define $\nu, \phi$ as  
\begin{align*}
\nu&=\sqrt{\frac{(p-1)(p+3)}{4}},
\\
\phi(x)&=\left( \frac{1}{\cosh{(\frac{p-1}{2}x)}}\right)^{\frac{p+1}{p-1}}.
\end{align*}
We note that 
\begin{align*}
\phi=\left(\frac{2}{p+1}\right)^{\frac{p+1}{2(p-1)}} Q^{\frac{p+1}{2}}
\end{align*}
holds. We introduce the following properties. In detail, see \cite{CGNT}.

\begin{lemma}\label{Lproperty}
\begin{enumerate}
\item $\mathcal{L}$ has a unique negative eigenvalue $-\nu^2$ and its eigenspace $(\mathcal{L}+\nu^2)^{-1}(\{0\})=\Span{\{\phi\}}$, and its kernel is $\Span{ \{Q^{\prime}\}}$.
\item There exists $c>0$ such that, for all $\varphi \in H^1(\mathbb{R})$,
\begin{align*}
\langle \mathcal{L}\varphi,\varphi\rangle\geq c\|\varphi\|_{H^1}^2-\frac{1}{c}\left(\langle \varphi,\phi\rangle^2+\langle\varphi,Q^{\prime}\rangle^2\right).
\end{align*}
\end{enumerate}
\end{lemma}

Now we consider the linearized operator around $(Q(\cdot-z),0)$, where $z\gg 1$. When we ignore the potential-soliton repulsion, the linearized evolution around $(Q(\cdot-z),0)$ is given by
\begin{align}\label{lode1}
\frac{d}{dt}v= 
\begin{pmatrix}
0 & 1
\\
-(-{\partial}_x^2+1-pQ^{p-1}(\cdot-z)) & -2\alpha
\end{pmatrix}
v.
\end{align}
We define $\nu^{\pm}$ and $Y^{\pm}$ as  
\begin{align*}
\nu^{\pm}&=-\alpha \pm \sqrt{\alpha^2+\nu^2},
\\
Y^{\pm}&=
\begin{pmatrix}
1
\\
\nu^{\pm}
\end{pmatrix}
\phi(\cdot-z).
\end{align*}
Since $\nu^+>0$, the solution $v=e^{\nu^+t}Y^+$ of \eqref{lode1} exhibits the exponential instability of $(Q(\cdot-z),0)$. In particular, we see that the presence of the damping $\alpha>0$ in the equation does not remove the exponential instability.

Second We recall properties of soliton interactions. In detail, see \cite{CMYZ}. 
Before we establish the lemma about soliton interactions, we define some notations. We assume that $z_1,z_2\in\mathbb{R}$ satisfy $|z_1-z_2|\gg 1$ and we define
\begin{align*}
Q_1(x)=Q(x-z_1),\ Q_2(x)=Q(x-z_2).
\end{align*}

\begin{lemma}\label{sil}
The following properties hold for $|z_1-z_2|\gg1$.
\begin{enumerate}
\item The following equation holds:
\begin{align}\label{c_qes1}
\int Q^pe^{-x}=2c_Q.
\end{align}
\item    For any $0<m^{\prime}<m$, 
\begin{align}
0<\int Q_1^mQ_2^m &\lesssim e^{-m^{\prime}|z_1-z_2|}, \label{si1}
\\
0<\int Q_1^mQ_2^{m^{\prime}}&\lesssim e^{-m^{\prime}|z_1-z_2|}. \label{si2}
\end{align}
\item For any $m>1$,
\begin{align*}
\left|\int Q_1Q_2^m-c_me^{-|z_1-z_2|}\right|&\lesssim e^{-\frac{3}{2}|z_1-z_2|},
\\
\left|\int ({\partial}_xQ_1)Q_2^m-\frac{z_1-z_2}{|z_1-z_2|}c_me^{-|z_1-z_2|}\right|&\lesssim e^{-\frac{3}{2}|z_1-z_2|} ,
\end{align*}
where $c_m=c_Q\int e^{-x}Q^m$.

\item The following estimate holds:
\begin{align*}
\left|\langle f(Q_1+Q_2)-f(Q_1)-f(Q_2),{\partial}_xQ_1\rangle-\frac{z_1-z_2}{|z_1-z_2|}\cdot 2c_Q^2e^{-|z_1-z_2|}\right|\lesssim e^{-\theta |z_1-z_2|},
\end{align*}
where $\theta=\min{(p-1,\frac{3}{2})}$.
\end{enumerate}
\end{lemma}

\begin{proof}
First we prove (1). Let $R>0$. By integration by parts, we obtain that as $R\to \infty$,
\begin{align*}
\int_{-R}^R Q^pe^{-x}&=\int_{-R}^R(-Q^{\prime \prime}+Q)e^{-x}
\\
&=c_Q+\int_{-R}^R(-Q^{\prime}+Q)e^{-x}+O(e^{-2R})
\\
&=2c_Q+O(e^{-2R}).
\end{align*}
Therefore we obtain \eqref{c_qes1}. (2), (3) and (4) are proved by \cite[Lemma 3.2]{CMY}, \cite[Lemma 2.1]{CMYZ} and \eqref{c_qes1}. 
\end{proof}

\subsection{Modulation of the center}

Now we assume that a solution $\vec{u}$ of \eqref{DNKG} is a single solitary wave or an even 2-solitary wave. We have some freedom for the choice of centers $z$. So we choose the center of the solitary wave.

\begin{lemma}\label{modulation}
\begin{enumerate}
\item Let $\sigma=\pm 1$. There exists $\delta>0$ small enough such that for any $\varphi=(\varphi_1,\varphi_2) \in \mathcal{H}$ and $z\in \mathbb{R}$ satisfying
\begin{align*}
|z|>\frac{1}{\delta} , \ \|\varphi_1-\sigma Q(\cdot-z)\|_{H^1}+\|\varphi_2\|_{L^2}\leq \delta,
\end{align*}
there exists a unique $\tilde{z}\in \mathbb{R}$ satisfying $|z-\tilde{z}|\leq \delta$ and
\begin{align*}
\int \{ \varphi_2+2\alpha(\varphi_1-\sigma Q(\cdot-\tilde{z}))\}Q^{\prime}(\cdot-\tilde{z})=0.
\end{align*} 
Moreover. $\tilde{z}$ is $C^1$ with respect to $(\varphi,z)$.

\item Let $\sigma=\pm 1$. There exists $\delta>0$ small enough such that for any $\varphi=(\varphi_1,\varphi_2) \in \mathcal{H}_{rad}$ and $z\in \mathbb{R}$ satisfying
\begin{align*}
z>\frac{1}{\delta} , \ \|\varphi_1-\sigma(Q(\cdot-z)+Q(\cdot+z))\|_{H^1}+\|\varphi_2\|_{L^2}\leq \delta,
\end{align*}
there exists a unique $\tilde{z}\in \mathbb{R}$ satisfying $|z-\tilde{z}|\leq \delta$ and
\begin{align*}
\int \{ \varphi_2+2\alpha\left(\varphi_1-\sigma(Q(\cdot-\tilde{z})+Q(\cdot+\tilde{z}))\right)\}Q^{\prime}(\cdot-\tilde{z})=0.
\end{align*} 
Moreover. $\tilde{z}$ is $C^1$ with respect to $(\varphi,z)$.

\end{enumerate}
\end{lemma}

\begin{proof}
We note that if $\varphi$ is an even function, for any $z>0$, 
\begin{align*}
\int \varphi Q^{\prime}(\cdot-z)=-\int \varphi Q^{\prime}(\cdot+z)
\end{align*}
holds. Then by \cite[Lemma 4.1]{IN} we complete the proof.

\end{proof}

By Lemma \ref{modulation}, if a global solution $\vec{u}$ is a single solitary wave, there exist $\sigma=\pm1$ and $z(t)$ such that
\begin{align*}
\begin{aligned}
\lim_{t\to \infty}\left(\|u(t)-\sigma Q(\cdot-z(t))\|_{H^1}+\|{\partial}_tu(t)\|_{L^2}\right)&=0,
\\
\lim_{t\to \infty}|z(t)|&=\infty,
\\
\int\{{\partial}_tu(t)+2\alpha(u(t)-\sigma Q(\cdot-z(t)))\}Q^{\prime}(\cdot-z(t))&=0,
\end{aligned}
\end{align*}
and if a global solution $\vec{u}$ is an even 2-solitary wave, there exist $\sigma=\pm1$ and $z(t)$ such that
\begin{align*}
\begin{aligned}
\vec{u}\ \mbox{is}\ \mbox{an}\ \mbox{even}\ \mbox{solution},
\\
\lim_{t\to \infty}\left(\|u(t)-\sigma(Q(\cdot-z(t))+Q(\cdot+z(t)))\|_{H^1}+\|{\partial}_tu(t)\|_{L^2}\right)&=0,
\\
\lim_{t\to \infty}z(t)&=\infty,
\\
\int\{{\partial}_tu(t)+2\alpha\left(u(t)-\sigma(Q(\cdot-z(t))+Q(\cdot+z(t)))\right)\}Q^{\prime}(\cdot-z(t))&=0.
\end{aligned}
\end{align*}

Furthermore by \eqref{invariance}, we only need the following proposition to prove Theorem \ref{centerdistance}.
\begin{proposition}\label{ia}
\begin{enumerate}
\item Let $\gamma<0$. If a global solution $\vec{u}$ satisfies 
\begin{align}\label{sa1}
\left\{
\begin{aligned}
&\lim_{t\to \infty}\left(\|u(t)-Q(\cdot-z(t))\|_{H^1}+\|{\partial}_tu(t)\|_{L^2}\right)=0,
\\
&\lim_{t\to \infty}z(t)=\infty,
\\
&\|u(0)-Q(\cdot-z(0))\|_{H^1}+\|{\partial}_tu(0)\|_{L^2}\ll1,\ z(0)\gg 1,
\\
&\int\{{\partial}_tu(t)+2\alpha(u(t)-Q(\cdot-z(t)))\}Q^{\prime}(\cdot-z(t))=0,
\end{aligned}
\r.
\end{align}
then $z(t)-\frac{1}{2}\log{t}\lesssim 1$.
\item Let $\gamma\leq-2$. If a global solution $\vec{u}$ satisfies
\begin{align}\label{sa2}
\left\{
\begin{aligned}
&\vec{u}\ \mbox{is}\ \mbox{an}\ \mbox{even}\ \mbox{solution},
\\
&\lim_{t\to \infty}\left(\|u(t)-Q(\cdot-z(t))-Q(\cdot+z(t))\|_{H^1}+\|{\partial}_tu(t)\|_{L^2}\right)=0,
\\
&\lim_{t\to\infty}z(t)=\infty,
\\
&\|u(0)-Q(\cdot-z(0))-Q(\cdot+z(0))\|_{H^1}+\|{\partial}_tu(0)\|_{L^2}\ll1,\ z(0)\gg 1,
\\
&\int\{{\partial}_tu(t)+2\alpha(u(t)-Q(\cdot-z(t))-Q(\cdot+z(t)))\}Q^{\prime}(\cdot-z(t))=0,
\end{aligned}
\r.
\end{align}
then $z(t)-\frac{1}{2}\log{t}\lesssim 1$.
\end{enumerate}

\end{proposition}

\subsection{Estimates of the center}

We assume that a solution of \eqref{DNKG} satisfies
\begin{align}\label{sa}
\left\{
\begin{aligned}
&\lim_{t\to \infty}\left(\|u(t)-Q(\cdot-z(t))-\sigma Q(\cdot+z(t))\|_{H^1}+\|{\partial}_tu(t)\|_{L^2}\right)=0,
\\
&\lim_{t\to\infty}z(t)=\infty,
\\
&\|u(0)-Q(\cdot-z(0))-\sigma Q(\cdot+z(0))\|_{H^1}+\|{\partial}_tu(0)\|_{L^2}\ll1,\ z(0)\gg 1,
\\
&\int\{{\partial}_tu(t)+2\alpha(u(t)-Q(\cdot-z(t))-\sigma Q(\cdot+z(t)))\}Q^{\prime}(\cdot-z(t))=0,
\end{aligned}
\r.
\end{align}
where $\sigma=0,1$. We define $Q_{\pm}$, $\vec{Q}_{\pm}$, and $\phi_{\pm}$ as
\begin{align*}
Q_{\pm}(x)=Q(x\mp z(t)) , \ \vec{Q}_{\pm}=
\begin{pmatrix}
Q_{\pm}\\
0   \\
\end{pmatrix}
,\ \phi_{\pm}(x)=\phi(x\mp z(t)).
\end{align*}
Furthermore we define $\vec{\epsilon}=(\epsilon,\eta)$ as 
\begin{align*}
\vec{\epsilon}=
\begin{pmatrix}
\epsilon\\
\eta \\
\end{pmatrix}
=\vec{u}-\vec{Q}_+-\sigma \vec{Q}_-.
\end{align*}
Then \eqref{DNKG} can be rewritten as
\begin{align}\label{lode}
\left\{
	\begin{aligned}
		{\partial}_t\epsilon&=\eta+z^{\prime}({\partial}_xQ_+-\sigma {\partial}_xQ_-),
		\\
		{\partial}_t\eta&=-\left\{-{\partial}_x^2\epsilon+\epsilon-(pQ_+^{p-1}+\sigma pQ_-^{p-1})\epsilon\right\}-2\alpha \eta 
		\\
		&\quad +\gamma \delta_0u+(f(u)-f(Q_+)-f(\sigma Q_-)-f^{\prime}(Q_+)\epsilon-f^{\prime}(\sigma Q_-)\epsilon).
	\end{aligned}
	\r.
\end{align}
In addition we define $\mathcal{L}_t$ and $R$ as
\begin{align*}
\mathcal{L}_t&=-{\partial}_x^2+1-pQ_+^{p-1}-\sigma pQ_-^{p-1},
\\
R&=f(u)-f(Q_+)-f(\sigma Q_-)-f^{\prime}(Q_+)\epsilon-f^{\prime}(\sigma Q_-)\epsilon.
\end{align*}
Then we rewrite \eqref{sa} and \eqref{lode} as
\begin{align}\label{sa*}
\left\{
\begin{aligned}
&\lim_{t\to \infty}\|\vec{\epsilon}(t)\|_{\mathcal{H}}=0 ,\ \lim_{t\to \infty} z(t)=\infty,\ \|\vec{\epsilon}(0)\|_{\mathcal{H}}\ll1,\ z(0)\gg 1,
\\
&\int(\eta+2\alpha \epsilon){\partial}_xQ_+=0,
\end{aligned}
\r.
\end{align}
and 
\begin{align}\label{lode2}
\frac{d}{dt}
\begin{pmatrix}
\epsilon
\\
\eta
\end{pmatrix}
=
\begin{pmatrix}
0 & 1
\\
-\mathcal{L}_t & -2\alpha
\end{pmatrix}
\begin{pmatrix}
\epsilon
\\
\eta
\end{pmatrix}
+
\begin{pmatrix}
z^{\prime}({\partial}_xQ_+-\sigma {\partial}_xQ_-)
\\
\gamma\delta_0u+R(t)
\end{pmatrix}
.
\end{align} 
Furthermore we define $R_1$ as
 \begin{align}\label{Res}
 R=p\sigma Q_+^{p-1}Q_-+R_1.
 \end{align}
First we give some estimates about $R$ and $R_1$. 
\begin{lemma}\label{R1es}
Let $\theta=\min{(2p-2,3)}$. Then the following estimates hold:
\begin{align}
\left|\int R_1{\partial}_xQ_+\right|&\lesssim e^{-\theta z}+\|\vec{\epsilon}\|_{\mathcal{H}}^2,\label{R1es1}
\\
\left|\int R_1\phi_+\right|&\lesssim e^{-\theta z}+\|\vec{\epsilon}\|_{\mathcal{H}}^2,\label{R1es2}
\\
\left|\int R{\partial}_xQ_+\right|&\lesssim e^{-2z}+\|\vec{\epsilon}\|_{\mathcal{H}}, \label{Res1}
\\
\left|\int R\phi_+\right|&\lesssim e^{-2z}+\|\vec{\epsilon}\|_{\mathcal{H}}^2,\label{Res2}
\\
\|R\|_{L^2}&\lesssim e^{-2z}+\|\vec{\epsilon}\|_{\mathcal{H}}^2. \label{Res3}
\end{align}
\end{lemma}
\begin{proof}
First $R_1$ satisfies
\begin{align*}
R_1&= f(u)-f(Q_+)-f(\sigma Q_-)-f^{\prime}(Q_+)\epsilon-f^{\prime}(\sigma Q_-)\epsilon-\sigma f^{\prime}(Q_+)Q_-
\\
&= f(Q_++\sigma Q_-+\epsilon)-f(Q_++\sigma Q_-)-f^{\prime}(Q_++\sigma Q_-)\epsilon
\\
&\quad +\left(f^{\prime}(Q_++\sigma Q_-)-f^{\prime}(Q_+)\right)\epsilon
\\
&\quad +f(Q_++\sigma Q_-)-f(Q_+)-f(\sigma Q_-)-f^{\prime}(Q_+)(\sigma Q_-)
\\
&\quad -f^{\prime}(\sigma Q_-)\epsilon.
\end{align*}
By direct computation we have
\begin{align*}
&|f(Q_++\sigma Q_-+\epsilon)-f(Q_++\sigma Q_-)-f^{\prime}(Q_++\sigma Q_-)\epsilon|\lesssim |\epsilon|^2+|\epsilon|^p,
\\
&|(f^{\prime}(Q_++\sigma Q_-)-f^{\prime}(Q_+))\epsilon|\lesssim \left(|Q_+|^{p-2}|Q_-|+|Q_-|^{p-1}\right)|\epsilon|,
\\
&|f(Q_++\sigma Q_-)-f(Q_+)-f(\sigma Q_-)-f^{\prime}(Q_+)(\sigma Q_-)|\lesssim |Q_+|^{p-2}|Q_-|^2+|Q_+||Q_-|^{p-1},
\\
&|f^{\prime}(\sigma Q_-)\epsilon|\lesssim|Q_-|^{p-1}|\epsilon|.
\end{align*}
Therefore by Lemma \ref{sil} and $e^{-2z}\|\vec{\epsilon}\|_{\mathcal{H}}\lesssim e^{-3z}+\|\vec{\epsilon}\|_{\mathcal{H}}^3$, we obtain \eqref{R1es1} and \eqref{R1es2}. Next we estimate \eqref{Res3}. $R$ satisfies
\begin{align*}
R&=f(Q_++\sigma Q_-+\epsilon)-f(Q_++\sigma Q_-)-f^{\prime}(Q_++\sigma Q_-)\epsilon
\\
&\quad+\left( f^{\prime}(Q_++\sigma Q_-)-f^{\prime}(Q_+)-f^{\prime}(Q_-)\right)\epsilon
\\
&\quad+f(Q_++\sigma Q_-)-f(Q_+)-f(\sigma Q_-)-f^{\prime}(Q_+)(\sigma Q_-).
\end{align*}
Furthermore, we have
\begin{align*}
|\left( f^{\prime}(Q_++\sigma Q_-)-f^{\prime}(Q_+)-f^{\prime}(Q_-)\right)\epsilon|\lesssim e^{-z}|\epsilon|,
\end{align*}
and \eqref{Res3} holds. Last, we prove \eqref{Res1} and \eqref{Res2}. By Lemma \ref{sil}, we have
\begin{align*}
\left|\int Q_+^{p-1}Q_-{\partial}_xQ_+\right|\lesssim e^{-2z},
\\
\left|\int Q_+^{p-1}Q_-\phi_+\right| \lesssim e^{-2z}.
\end{align*}
Thus we obtain \eqref{Res1} and \eqref{Res2}. 



\end{proof}

Next we estimate the distance between the origin and the center of the solitary wave. 

\begin{proposition}\label{zpes}
Let $\vec{u}$ be a global solution of \eqref{DNKG} satisfying \eqref{sa*} for some $z(t)$ and $\sigma=0,1$. Then $z^{\prime}$ satisfies 
\begin{align}\label{zes}
\begin{aligned}
2\alpha \|Q^{\prime}\|_{L^2}^2z^{\prime}&=-\gamma c_Qe^{-z}u(t,0)-2\sigma c_Q^2 e^{-2z}-R_2+O(e^{-\theta z}+\|\vec{\epsilon}\|_{\mathcal{H}}^3)
\\
&=\{-\gamma(1+\sigma)-2\sigma\}c_Q^2e^{-2z}-\gamma c_Qe^{-z}\epsilon(t,0)-R_2+O(e^{-\theta z}+\|\vec{\epsilon}\|_{\mathcal{H}}^3),
\end{aligned}
\end{align}
where $\theta=\min{(2p-2,3)}$ and 
\begin{align*}
R_2=\int R_1{\partial}_xQ_+.
\end{align*}
\end{proposition}

\begin{proof}
Differentiating  the second equation of \eqref{sa*}, we obtain
\begin{align}\label{zpeq}
\begin{aligned}
0&=\int({\partial}_t\eta+2\alpha{\partial}_t\epsilon){\partial}_xQ_+-z^{\prime}\int (\eta+2\alpha\epsilon){\partial}_x^2Q_+
\\
&=z^{\prime}\left(2\alpha\|Q^{\prime}\|_{L^2}^2-2\alpha \sigma \int {\partial}_xQ_-{\partial}_xQ_+-\int (\eta+2\alpha\epsilon){\partial}_x^2Q_+\right)
\\
&\quad +\gamma\int\delta_0(u\cdot {\partial}_xQ_+)+\sigma p\int Q_-^{p-1}{\partial}_xQ_+\epsilon+\int R{\partial}_xQ_+.
\end{aligned}
\end{align}
By Lemma \ref{sil} and \eqref{Res}, we have
\begin{align}\label{Res4}
\begin{aligned}
\int R{\partial}_xQ_+&=\sigma \int\frac{d}{dx}(Q_+^p)Q_-+R_2
\\
&=2\sigma c_Q^2e^{-2z}+R_2.
\end{aligned}
\end{align}
Furthermore by \eqref{Qes}, \eqref{si1}, \eqref{zpeq} and \eqref{Res4}, we have
\begin{align*}
0&=z^{\prime}\left(2\alpha\|Q^{\prime}\|_{L^2}^2+O(e^{-\frac{3}{2}z}+\|\vec{\epsilon}\|_{\mathcal{H}})\right)
\\
&\quad+\gamma c_Qe^{-z}u(t,0)+2\sigma c_Q^2e^{-2z}+R_2+O(\|\vec{\epsilon}\|_{\mathcal{H}}^3+e^{-3z}).
\end{align*}
Therefore we obtain \eqref{zes}.

\end{proof}

\begin{remark}
By \eqref{R1es1}, $R_2$ satisfies 
\begin{align}\label{R2es}
|R_2|\lesssim \|\vec{\epsilon}\|_{\mathcal{H}}^2+e^{-\theta z},
\end{align}
where $\theta=\min{(2p-2,3)}$.
\end{remark}
By Proposition \ref{zpes} and \eqref{R2es}, $z$ satisfies
\begin{align}\label{zpes1}
|z^{\prime}|\lesssim e^{-2z}+\|\vec{\epsilon}\|_{\mathcal{H}}^2.
\end{align}

\begin{remark}
To prove Theorem \ref{centerdistance}, we only need \eqref{zpes1}. However, by Proposition \ref{zpes} and \eqref{R2es} we conjecture that if $\sigma=1, \gamma<-1$ or $\sigma=0, \gamma<0$, $z$ satisfies $z\sim \frac{1}{2}\log{t}$.
\end{remark}

Next we define some notations to extract eigenmodes with real eigenvalues around the solitary wave. We define $a_+, a_-, a_0$ as
\begin{align*}
a_{\pm}(t)&=\int(\eta-\nu^{\mp}\epsilon)\phi_+,
\\
a_0(t)&=\int\eta {\partial}_xQ_+.
\end{align*}
\begin{proposition}\label{aest}
Let $\vec{u}$ be a solution of \eqref{DNKG} satisfying \eqref{sa*} for some $z(t)$ and $\sigma=0,1$. Then $a_{\pm}$ and $a_0$ satisfy
\begin{align}
\left|\frac{d}{dt}a_{\pm}-\nu^{\pm}a_{\pm}\right|&\lesssim e^{-2z}+\|\vec{\epsilon}\|_{\mathcal{H}}^2,\label{apmes}
\\
\left|\frac{d}{dt}a_0+2\alpha a_0\right|&\lesssim  e^{-2z}+\|\vec{\epsilon}\|_{\mathcal{H}}^2.\label{a0es}
\end{align} 
\end{proposition}

\begin{proof}
As the proof of Proposition \ref{zpes}, we compute $a_{\pm}$ and $a_0$. First, we have
\begin{align}\label{apmes1}
\frac{d}{dt}a_{\pm}&=\int ({\partial}_t\eta-\nu^{\mp}{\partial}_t\epsilon)\phi_+-z^{\prime}\int(\eta-\nu^{\mp}\epsilon){\partial}_x\phi_+.
\end{align}
We note that 
\begin{align}\label{apmes2}
\begin{aligned}
 {\partial}_t\eta+\nu^{\mp}{\partial}_t\epsilon&=-(-{\partial}_x^2\epsilon+\epsilon-pQ_+^{p-1}\epsilon)-(2\alpha+\nu^{\mp})\eta
\\
&\quad+\sigma pQ_-^{p-1}\epsilon+\gamma \delta_0u+R(t)-\nu^{\mp}z^{\prime}({\partial}_xQ_+^{\prime}-\sigma {\partial}_xQ_-).
\end{aligned}
\end{align}
By \eqref{Res2}, \eqref{zpes1}, \eqref{apmes1}, and \eqref{apmes2}, we obtain 
\begin{align*}
\frac{d}{dt}a_{\pm}&=\int \left\{-(-{\partial}_x^2\epsilon+\epsilon-pQ_+^{p-1}\epsilon)-(2\alpha+\nu^{\mp})\eta\right\} \phi_++O(\|\vec{\epsilon}\|_{\mathcal{H}}^2+e^{-2z}),
\\
&\quad\int \left\{-(-{\partial}_x^2\epsilon+\epsilon-pQ_+^{p-1}\epsilon)-(2\alpha+\nu^{\mp})\eta\right\} \phi_+
\\
&=\int \epsilon\cdot\left(-{\partial}_x^2\phi_++\phi_+-pQ_+^{p-1}\phi_+\right)-(2\alpha+\nu^{\mp})\phi_+
\\
&=\int (\nu^2\epsilon+\nu^{\pm}\eta)\phi_+
\\
&=\nu^{\pm}a_{\pm}.
\end{align*}
Therefore we obtain \eqref{apmes}. Second we have
\begin{align*}
\frac{d}{dt}a_0=-2\alpha\int \eta {\partial}_xQ_++\int (p\sigma Q_-^{p-1}\epsilon+\gamma\delta_0u+R){\partial}_xQ_+.
\end{align*}
By the same argument as the proof of Proposition \ref{zpes}, we have
\begin{align*}
\left|\int (p\sigma Q_-^{p-1}\epsilon+\gamma\delta_0u+R){\partial}_xQ_+\right|\lesssim e^{-2z}+\|\vec{\epsilon}\|_{\mathcal{H}}^2.
\end{align*}
Thus we obtain \eqref{a0es}.

\end{proof}

By Proposition \ref{aest} and $|a_0|+|a_+|+|a_-|\lesssim \|\vec{\epsilon}\|_{\mathcal{H}}$, we have
\begin{align}
\left|\frac{d}{dt}(a_{\pm})^2-2\nu^{\pm}(a_{\pm})^2\right|&\lesssim |a_{\pm}|(e^{-2z}+\|\vec{\epsilon}\|_{\mathcal{H}}^2),\label{apmes3}
\\
\left|\frac{d}{dt}(a_0)^2+4\alpha (a_0)^2\right|&\lesssim e^{-3z}+\|\vec{\epsilon}\|_{\mathcal{H}}^3. \label{a0es2}
\end{align}

\subsection{Energy estimates}

For $\mu>0$ small enough we denote $\rho=2\alpha-\mu$. We define $\mathcal{E}$ as 
\begin{align}\label{Edef}
\mathcal{E}(t)=\frac{1}{2}\int \{({\partial}_x\epsilon)^2+(1-\rho\mu)\epsilon^2+(\eta+\mu\epsilon)^2-(pQ_+^{p-1}+p\sigma Q_-^{p-1})\epsilon^2\}-\frac{\gamma}{2}|u(t,0)|^2.
\end{align}
\begin{lemma}
Let $\vec{u}$ be a solution of \eqref{DNKG} satisfying \eqref{sa*} for some $z(t)$ and $\sigma=0,1$. Then for $t>0$, 
\begin{align}\label{Ees11}
\mathcal{E}^{\prime}(t)=-2\mu\mathcal{E}(t)-(\rho-\mu)\|\eta+\mu\epsilon\|_{L^2}^2-\mu\gamma c_Q(1+\sigma)e^{-z}u(t,0)+O(e^{-3z}+\|\vec{\epsilon}\|_{\mathcal{H}}^3).
\end{align}
\end{lemma}

\begin{proof}
By formal calculation, we have
\begin{align*}
\mathcal{E}^{\prime}(t)&=\int {\partial}_x\epsilon{\partial}_x{\partial}_t\epsilon+(1-\rho\mu)\epsilon{\partial}_t\epsilon+(\eta+\mu\epsilon)({\partial}_t\eta+\mu{\partial}_t\epsilon)-(pQ_+^{p-1}+p\sigma Q_-^{p-1})\epsilon{\partial}_t\epsilon
\\
&\quad+\frac{p-1}{2}z^{\prime}\left\{\int (Q_+^{p-2}{\partial}_xQ_+-{\sigma}Q_-^{p-2}{\partial}_xQ_-)\epsilon^2\right\}-\gamma u(t,0){\partial}_tu(t,0).
\end{align*}
Furthermore by \eqref{lode} and integration by parts, we have
\begin{align*}
&\quad\int {\partial}_x\epsilon{\partial}_x{\partial}_t\epsilon+(1-\rho\mu)\epsilon{\partial}_t\epsilon+(\eta+\mu\epsilon)({\partial}_t\eta+\mu{\partial}_t\epsilon)-(pQ_+^{p-1}+p\sigma Q_-^{p-1})\epsilon{\partial}_t\epsilon
\\
&=\int (-{\partial}_x^2\epsilon+(1-\rho\mu)\epsilon+\mu(\eta+\mu\epsilon)-(pQ_+^{p-1}+p\sigma Q_-^{p-1})\epsilon){\partial}_t\epsilon+(\eta+\mu\epsilon){\partial}_t\eta,
\\
&\quad\int (-{\partial}_x^2\epsilon+(1-\rho\mu)\epsilon+\mu(\eta+\mu\epsilon)-(pQ_+^{p-1}+p\sigma Q_-^{p-1})\epsilon){\partial}_t\epsilon
\\
&=\int (\mathcal{L}_t\epsilon-\mu(\rho-\mu)\epsilon+\mu\eta){\partial}_t\epsilon,
\\
&\quad\int (\eta+\mu\epsilon){\partial}_t\eta
\\
&=\int (\eta+\mu\epsilon)(-\mathcal{L}_t\epsilon-2\alpha\eta+\gamma\delta_0u+R(t)).
\end{align*}
By \eqref{Res3} and \eqref{zpes1},  we have
\begin{align*}
&\int (\mathcal{L}_t\epsilon-\mu(\rho-\mu)\epsilon+\mu\eta){\partial}_t\epsilon=\int (\mathcal{L}_t\epsilon-\mu(\rho-\mu)\epsilon+\mu\eta)\eta+O(e^{-3z}+\|\vec{\epsilon}\|_{\mathcal{H}}^3),
\\
&\int  (\eta+\mu\epsilon)(\gamma\delta_0u+R(t))=\gamma{\partial}_tu(t,0)u(t,0)+\mu\gamma\epsilon(t,0)u(t,0)+O(e^{-3z}+\|\vec{\epsilon}\|_{\mathcal{H}}^3).
\end{align*}
Gathering these estimates we obtain 
\begin{align*}
\mathcal{E}^{\prime}(t)&=\int\left[ \{(\mathcal{L}_t\epsilon-\mu(\rho-\mu)\epsilon+\mu\eta)\eta\}+(\eta+\mu\epsilon)(-\mathcal{L}_t\epsilon-2\alpha\eta)\right]
\\
&\quad+\mu\gamma\epsilon(t,0)u(t,0)+O(e^{-3z}+\|\vec{\epsilon}\|_{\mathcal{H}}^3)
\\
&=-2\mu\mathcal{E}(t)-(\rho-\mu)\|\eta+\mu\epsilon\|_{L^2}^2+\mu\gamma (\epsilon(t,0)-u(t,0))u(t,0)+O(e^{-3z}+\|\vec{\epsilon}\|_{\mathcal{H}}^3)
\\
&=-2\mu\mathcal{E}(t)-(\rho-\mu)\|\eta+\mu\epsilon\|_{L^2}^2-\mu\gamma c_Q(1+\sigma)e^{-z}u(t,0)+O(e^{-3z}+\|\vec{\epsilon}\|_{\mathcal{H}}^3),
\end{align*}
and we complete the proof.
\end{proof}

Furthermore by \eqref{zes}, we obtain
\begin{align}
\begin{aligned}\label{Ees1}
\mathcal{E}^{\prime}&=-2\mu \mathcal{E}-\mu\gamma c_Q^2(1+\sigma)^2e^{-2z}-\mu\gamma c_Q(1+\sigma )e^{-z}\epsilon(t,0)
\\
&\quad-(\rho-\mu)\|\eta+\mu\epsilon\|_{L^2}^2+O(e^{-3z}+\|\vec{\epsilon}\|_{\mathcal{H}}^3)
\\
&=-2\mu\mathcal{E}(t)+C_{\circ}z^{\prime}+C_{\bullet}e^{-2z}-\mu(1+\sigma)R_2
\\
&\quad-(\rho-\mu)\|\eta+\mu\epsilon\|_{L^2}^2+O(e^{-3z}+\|\vec{\epsilon}\|_{\mathcal{H}}^3),
\end{aligned}
\end{align}
where 
\begin{align*}
C_{\circ}=2(1+\sigma)\alpha\|Q^{\prime}\|_{L^2}^2\mu,\ C_{\bullet}=2\sigma(1+\sigma)c_Q^2\mu .
\end{align*}
We define
\begin{align*}
b(t)&=\left(\int \epsilon {\partial}_xQ_-\right)^2+\left(\int \epsilon \phi_-\right)^2.
\end{align*}

\begin{lemma}\label{enees}
Let $\vec{u}$ be a solution of \eqref{DNKG} satisfying \eqref{sa*} for some $z(t)$ and $\sigma=0,1$.  Then there exists $C>0$ large enough such that for all $t\geq0$,
\begin{align}\label{enees1}
C\left\{a_+(t)^2+a_-(t)^2+a_0(t)^2+\sigma b(t)\right\}+\mathcal{E}(t)+\frac{\gamma}{2}|u(t,0)|^2\geq \frac{1}{C}\|\vec{\epsilon}(t)\|_{\mathcal{H}}^2.
\end{align}
Furthermore if $\gamma<0$, there exists $\tilde{C}>0$ large enough such that for all $t\geq0$
\begin{align}\label{Eupperbound}
\begin{split}
&\quad\mathcal{E}+C\left(a_+^2+a_-^2+a_0^2+\sigma b\right)+e^{-2z}
\\
&\leq \tilde{C}\left\{ (E_{\gamma}(\vec{u})-(1+\sigma)J_0(Q))+(a_+^2+a_-^2+a_0^2+\sigma b^2)+e^{-2z}\right\}.
\end{split}
\end{align}
\end{lemma}

\begin{remark}
$C$ and $\tilde{C}$ are independent of $\mu>0$.
\end{remark}

In detail, see \cite[Lemma 2.6]{CMYZ}.
\begin{proof}
First we note that
\begin{align*}
\langle \epsilon,\phi_+\rangle=\frac{a_+-a_-}{\nu^+-\nu^-},\ \langle \epsilon,{\partial}_xQ_+\rangle=-\frac{a_0}{2\alpha}
\end{align*}
hold. Therefore we have
\begin{align}\label{asim}
\langle \epsilon,\phi_+\rangle^2+\langle \epsilon,{\partial}_xQ_+\rangle^2\lesssim a_+^2+a_-^2+a_0^2.
\end{align}
When $\sigma=0$, by Lemma \ref{Lproperty} and $\mu>0$ small enough, there exists $C>0$ large enough such that \eqref{enees1} is satisfied.

We consider $\sigma=1$. Let $\chi$ be a smooth function satisfying the following properties 
\begin{align*}
\chi=1\ \mbox{on}\ [0,1],\ \ \chi=0\ \mbox{on}\ [2,\infty),\ \ \chi^{\prime}\leq 0\ \mbox{on}\ \mathbb{R}.  
\end{align*}
For $\lambda=\frac{|z|}{2}\gg 1$, let 
\begin{align*}
\chi_+(x)=\chi\left(\frac{|x-z|}{\lambda}\right),\ \chi_-(x)=(1-\chi_+^2(x))^{\frac{1}{2}}.
\end{align*}
We define $\epsilon_{\pm}=\epsilon\cdot\chi_{\pm}$. Then we have
\begin{align}\label{chi}
\epsilon^2=\epsilon_+^2+\epsilon_-^2.
\end{align}
Therefore we have
\begin{align}
 \|\epsilon_+\|_{L^2}^2+\|\epsilon_-\|_{L^2}^2=\|\epsilon\|_{L^2}^2,\label{pmes1}
\\
\|{\partial}_x\epsilon\|_{L^2}^2=\|{\partial}_x\epsilon_+\|_{L^2}^2+\|{\partial}_x\epsilon_-\|_{L^2}^2+O\left(\frac{\|\epsilon\|_{H^1}^2}{|z|}\right),
\\
|\langle \epsilon-\epsilon_-, {\partial}_xQ_-\rangle|+|\langle \epsilon-\epsilon_-,\phi_-\rangle|\lesssim e^{-z}\|\epsilon\|_{L^2},\label{chies1}
\\
|\langle \epsilon-\epsilon_+,{\partial}_xQ_+\rangle|+|\langle \epsilon-\epsilon_+,\phi_+\rangle|\lesssim e^{-\frac{z}{2}}\|\epsilon\|_{L^2}. \label{chies2}
\end{align}
By Lemma \ref{Lproperty}, we have
\begin{align}
\langle \mathcal{L}_t\epsilon_+,\epsilon_+\rangle+\langle pQ_-^{p-1},\epsilon_+^2\rangle\geq c\|\epsilon_+\|_{H^1}^2-\frac{1}{c}\left(\langle \epsilon_+,\phi_+\rangle^2+\langle \epsilon_+,{\partial}_xQ_+\rangle^2\right),
\\
\langle \mathcal{L}_t\epsilon_-,\epsilon_-\rangle+\langle pQ_+^{p-1},\epsilon_-^2\rangle\geq c\|\epsilon_-\|_{H^1}^2-\frac{1}{c}\left(\langle \epsilon_-,\phi_-\rangle^2+\langle \epsilon_-,{\partial}_xQ_-\rangle^2\right).
\end{align}
Furthermore
\begin{align}\label{pmes2}
|\langle pQ_-^{p-1},\epsilon_+^2\rangle|+|\langle pQ_+^{p-1},\epsilon_-^2\rangle|\lesssim e^{-\frac{z}{2}}\|\epsilon\|_{L^2}^2
\end{align}
holds. We note that by \eqref{chi} we have 
\begin{align}\label{lchi}
\langle \mathcal{L}_t\epsilon,\epsilon\rangle=\langle \mathcal{L}_t\epsilon_+,\epsilon_+\rangle+\langle \mathcal{L}_t\epsilon_-,\epsilon_-\rangle+O\left(\frac{\|\epsilon\|_{H^1}^2}{|z|}\right).
\end{align}
Gathering \eqref{pmes1}-\eqref{lchi}, we obtain 
\begin{align*}\label{ltes1}
\langle\mathcal{L}_t\epsilon,\epsilon\rangle\geq \frac{c}{2}\|\epsilon\|_{H^1}^2-\frac{1}{c}\sum_{\ast=\pm}\left(\langle\epsilon,{\partial}_xQ_{\ast}\rangle^2+\langle\epsilon,\phi_{\ast}\rangle^2\right).
\end{align*}
Furthermore by  \eqref{asim} and $\mu>0$ small enough, there exists $C>0$ large enough such that \eqref{enees1} is satisfied.

Next we estimate \eqref{Eupperbound}. We denote $E=J_0(Q)$. By direct computation, we have
\begin{align*}
E_{\gamma}(\vec{u})&=(1+\sigma)E-\sigma \int Q_+^pQ_-+\frac{1}{2}\langle \mathcal{L}_t\epsilon,\epsilon\rangle+\frac{1}{2}\|\eta\|_{L^2}^2-\frac{\gamma}{2}|u(t,0)|^2
\\
&\quad-\int \left\{f(Q_++\sigma Q_-)-f(Q_+)-f(\sigma Q_-)\right\}\epsilon
\\
&\quad-\frac{1}{2}\int \left\{f^{\prime}(Q_++\sigma Q_-)-f^{\prime}(Q_+)-f^{\prime}(\sigma Q_-)\right\}\epsilon^2
\\
&\quad-\frac{1}{p+1}\int \left\{(Q_++\sigma Q_-)^{p+1}-Q_+^{p+1}-\sigma Q_-^{p+1}\right\}+O(\|\vec{\epsilon}\|_{\mathcal{H}}^3+e^{-2z}).
\end{align*}
We note that 
\begin{align*}
\left|\left\{f^{\prime}(Q_++\sigma Q_-)-f^{\prime}(Q_+)-f^{\prime}(\sigma Q_-)\right\}\epsilon^2\right|\lesssim \left(|Q_+^{p-2}Q_-|+|Q_+Q_-^{p-2}|\right)\epsilon^2 
\end{align*}
holds. Therefore we obtain
\begin{align}
E_{\gamma}(\vec{u})&=(1+\sigma)E+\frac{1}{2}\langle \mathcal{L}_t\epsilon,\epsilon\rangle+\frac{1}{2}\|\eta\|_{L^2}^2-\frac{\gamma}{2}|u(t,0)|^2+O(\|\vec{\epsilon}\|_{\mathcal{H}}^{\tau}+e^{-2z}),
\end{align}
where $\tau=\frac{4}{2-\min{(2(p-2),\frac{2}{3})} }>2$. By \eqref{Edef}, we have
\begin{align*}
E_{\gamma}(\vec{u})-(1+\sigma)E=\mathcal{E}+\frac{\mu}{2}\left\{(\rho-\mu)\|\epsilon\|_{L^2}^2-2\langle \epsilon,\eta\rangle\right\}+O(\|\vec{\epsilon}\|_{\mathcal{H}}^{\tau}+e^{-2z}).
\end{align*}
We recall that $\mu>0$ is small enough and that $\gamma$ is negative. By \eqref{enees1}, we have
\begin{align}\label{Enees5}
E_{\gamma}(\vec{u})-(1+\sigma)E&\geq \frac{1}{2}\mathcal{E}-\frac{C}{2}\left(a_+^2+a_-^2+a_0^2+\sigma b^2\right)-C^{\prime}e^{-2z}
\end{align}
for some $C^{\prime}>0$. By \eqref{Enees5}, we obtain \eqref{Eupperbound}.
\end{proof}

Now we assume that $\vec{u}$ satisfies \eqref{sa1} or \eqref{sa2}. Then $\vec{u}$ satisfies \eqref{sa*} for some $z(t)$ and $\sigma=0,1$. Furthermore, when $\vec{u}$ satisfies \eqref{sa2}, we have
\begin{align*}
\langle \epsilon,{\partial}_xQ_-\rangle=-\langle \epsilon,{\partial}_xQ_+\rangle,\ \langle\epsilon,\phi_+\rangle=\langle\epsilon,\phi_-\rangle,
\end{align*}
because $\epsilon$ is even. Thus we get the following result.

\begin{corollary}\label{e}
We assume that $\vec{u}$ satisfies \eqref{sa1} or \eqref{sa2}. Then there exists $C>0$ such that 
\begin{align}\label{enees2}
C\left\{a_+(t)^2+a_-(t)^2+a_0(t)^2\right\}+\mathcal{E}(t)+\frac{\gamma}{2}|u(t,0)|^2\gtrsim \|\vec{\epsilon}\|_{\mathcal{H}}^2.
\end{align}
Furthermore if $\gamma<0$, there exists $\tilde{C}>0$ large enough such that 
\begin{align*}
&\quad\mathcal{E}+C\left(a_+^2+a_-^2+a_0^2\right)+e^{-2z}
\\
&\leq\tilde{C}\left\{ (E_{\gamma}(\vec{u})-2J_0(Q))+(a_+^2+a_-^2+a_0^2)+e^{-2z}\right\}.
\end{align*}
\end{corollary}

Let $L$ be large enough. We define $\mathcal{G}$ as 
\begin{align}\label{Gdef}
\mathcal{G}(t)=\mathcal{E}(t)+L\left(a_-^2(t)+a_0^2(t)\right).
\end{align}
By Corollary \ref{e}, if $\gamma<0$ and $\vec{u}$ satisfies \eqref{sa1} or \eqref{sa2}, we have
\begin{align}\label{Ges2}
c_{\mathcal{G}}\|\vec{\epsilon}\|_{\mathcal{H}}^2\leq \mathcal{G}+La_+^2&\leq C_{\mathcal{G}}\left\{(E_{\gamma}(\vec{u})-(1+\sigma)J_0(Q))+a_+^2+a_-^2+a_0^2+e^{-2z}\right\},
\end{align}
for some constants $c_{\mathcal{G}}>0$ and $C_{\mathcal{G}}>0$. In addition $\mathcal{G}$ satisfies
\begin{align}\label{Ges3}
|\mathcal{G}|\lesssim \|\vec{\epsilon}\|_{\mathcal{H}}^2+e^{-2z}.
\end{align}
\begin{lemma}
Let $\gamma<0$. We assume that $\vec{u}$ satisfies \eqref{sa1} or \eqref{sa2}. Then there exist $L_1>0$ and $\delta>0$ such that 
\begin{align}\label{Ges}
\mathcal{G}^{\prime}\leq -\delta \mathcal{G}+L_1a_+^2+\frac{1}{\delta}\left(e^{-2z}+\|\vec{\epsilon}\|_{\mathcal{H}}^3\right).
\end{align}
\end{lemma}
\begin{proof}
By \eqref{apmes3}, \eqref{a0es2}, \eqref{Ees1} and \eqref{Gdef}, we have
\begin{align*}
\mathcal{G}^{\prime}&=-2\mu\mathcal{E}-\mu \gamma c_Q(1+\sigma)e^{-z}\epsilon(t,0)-(\rho-\mu)\|\eta+\mu\epsilon\|_{L^2}^2
\\
&\quad+2L\nu^-a_-^2-4\alpha La_0^2+O(e^{-2z}+\|\vec{\epsilon}\|_{\mathcal{H}}^3).
\end{align*}
Now we choose $\hat{\delta}>0$ small so that 
\begin{align*}
2(\mu-\hat{\delta})\mathcal{E}(t)+2\mu L\left(a_-^2(t)+a_+^2(t)+a_0^2(t)\right)+\frac{\gamma}{2}|u(t,0)|^2\geq \hat{\delta}|\epsilon(t,0)|^2,
\\
\nu^-+\mu<-\hat{\delta}, \ -2\alpha+\mu<-\hat{\delta}.
\end{align*}
We note that the above estimate holds by \eqref{enees2}. Then there exists $C>0$ such that
\begin{align*}
\mathcal{G}^{\prime}&\leq -2\mu \mathcal{E}+\hat{\delta}|\epsilon(t,0)|^2+2L\nu^-a_-^2-4\alpha La_0^2+C\left(e^{-2z}+\|\vec{\epsilon}\|_{\mathcal{H}}^3\right)
\\
&\leq -2\hat{\delta}\mathcal{G}+2\mu La_+^2+C\left(e^{-2z}+\|\vec{\epsilon}\|_{\mathcal{H}}^3\right).
\end{align*}
Therefore by choosing $L_1=2\mu L$ and  $\delta=\min{(2\hat{\delta}, \frac{1}{C})}$, we obtain \eqref{Ges}.

\end{proof}

\subsection{Estimate of the instability term}

In this subsection, we estimate the instability term $a_+$.

\begin{proposition}\label{instaes}
Let $\gamma<0$. We assume that $\vec{u}$ satisfies \eqref{sa1} or \eqref{sa2}. Then for any $M>0$, there exists $T_M>0$ such that for all $t\geq T_M$
\begin{align}\label{apes1}
|a_+(t)|\leq \frac{1}{M}\left(\mathcal{G}+e^{-2z}\right)^{\frac{1}{2}}.
\end{align}
\end{proposition}

\begin{proof}
We fix $M>0$. We may assume that there exist $0<\delta_M\ll 1$ and $T_M>0$ such that for $t\geq T_M$ 
\begin{align}\label{smalles}
e^{-z(t)}+|\mathcal{G}(t)|+|a_+(t)|&<\delta_M.
\end{align}
In addition we assume that there exists $T_1>T_M$ satisfying
\begin{align*}
\frac{1}{M}\left(\mathcal{G}(T_1)+e^{-2z(T_1)}\right)^{\frac{1}{2}}<|a_+(T_1)|.
\end{align*}
We introduce the following bootstrap estimate
\begin{align}\label{bs1}
\frac{1}{M_1}\left(\mathcal{G}+e^{-2z}\right)^{\frac{1}{2}}\leq |a_+|\leq \delta_M,
\end{align}
where $M_1$ satisfies
\begin{align*}\label{M_1es}
(M_1^2-C_{\mathcal{G}})\nu^+-2C_{\mathcal{G}}-1>0,\ M_1>M.
\end{align*}
We define $T_2\geq T_1$ as 
\begin{align*}
T_2=\sup{\{t\in [T_1,\infty)\ \mbox{such}\ \mbox{that}\ \mbox{holds}\ \eqref{bs1}\ \mbox{on}\ [T_1,t]\}}.
\end{align*}
Since $\delta_M$ is sufficiently small, $a_+$ satisfies for $t\geq T_M$
\begin{align}
\frac{d}{dt}|a_+|^2\geq \nu^+|a_+|^2.
\end{align}
Furthermore by \eqref{zpes1}, \eqref{apmes3}, and \eqref{a0es2} there exists $C_1>0$ such that 
\begin{align*}
\frac{d}{dt}a_0^2\leq C_1\left(\|\vec{\epsilon}\|_{\mathcal{H}}^3+e^{-3z}\right),
\\
\frac{d}{dt}a_-^2\leq C_1\left(\|\vec{\epsilon}\|_{\mathcal{H}}^3+e^{-3z}\right),
\\
|z^{\prime}|\leq C_1\left(\|\vec{\epsilon}\|_{\mathcal{H}}^2+e^{-2z}\right).
\end{align*}
Since $\delta_M$ is small enough, we have
\begin{align}
\frac{d}{dt}\left\{(a_0)^2+(a_-^2)\right\}\leq a_+^2, \label{ama0es}
\\
\left|\frac{d}{dt}e^{-2z}\right|\leq a_+^2. \label{zesap}
\end{align}
We define 
\begin{align*}
\mathcal{I}=M_1^2a_+^2-C_{\mathcal{G}}\left(E(\vec{u})-2J_0(Q)+a_+^2+a_-^2+a_0^2+e^{-2z}\right)-e^{-2z}.
\end{align*}
Then by \eqref{energydecay}, \eqref{ama0es}, and \eqref{zesap}  we have
\begin{align*}
\mathcal{I}^{\prime}\geq \left\{(M_1^2-C_{\mathcal{G}})\nu^+-2C_{\mathcal{G}}-1\right\}a_+^2>0.
\end{align*}
Therefore as long as \eqref{bs1} holds, we have 
\begin{align*}
e^{-2z}+C_{\mathcal{G}}\left(E(\vec{u})-2J_0(Q)+a_+^2+a_-^2+a_0^2+e^{-2z}\right)<M_1^2a_+^2.
\end{align*}
Thus by \eqref{Ges2}, we obtain
\begin{align}\label{bs2}
\mathcal{G}+e^{-2z}<M_1^2a_+^2.
\end{align}
This means that as long as \eqref{bs1} holds, we have \eqref{bs2}. Therefore for $T_1\leq t\leq T_2$, \eqref{bs2} holds and $a_+^2$ increases exponentially. Thus $a_+(T_3)=\delta_M$ holds for some $T_3>0$, which contradicts \eqref{smalles}. Therefore for all $t\geq T_M$, 
\begin{align*}
|a_+|\leq \frac{1}{M}\left(\mathcal{G}+e^{-2z}\right)^{\frac{1}{2}}.
\end{align*}
\end{proof}
We note that since $\|\vec{\epsilon}(0)\|_{\mathcal{H}}$ is small enough we can replace $T_M$ with $0$ in Proposition \ref{instaes}.

\subsection{Proof of Theorem \ref{centerdistance}}
In this subsection we prove Theorem \ref{centerdistance}. First we organize some estimates of $\mathcal{G}$. Gathering \eqref{Ges2}-\eqref{apes1}, there exist $C>0$ and $\delta>0$ satisfying
\begin{align}
\mathcal{G}^{\prime}\leq -\delta\mathcal{G}+Ce^{-2z},\label{Ges5}
\\
\mathcal{G}+e^{-2z}\sim \|\vec{\epsilon}\|_{\mathcal{H}}^2+e^{-2z}.\label{Ges6}
\end{align}
Therefore by \eqref{zpes1}, there exists $\hat{C}>0$ such that 
\begin{align}\label{zesfinal}
z^{\prime}\leq \hat{C}\left(\mathcal{G}+e^{-2z}\right).
\end{align}

Before we prove Theorem \ref{centerdistance}, we prove Proposition \ref{ia}.

\begin{proof}[Proof of Proposition \ref{ia}]
We define $w(t)=e^{2z(t)}$. Then by \eqref{zesfinal} we have
\begin{align}
w^{\prime}\leq 2\hat{C}\mathcal{G}w+2\hat{C}.
\end{align}
Since $\lim_{t\to \infty} w(t)=\infty$ and $\lim_{t\to \infty}\mathcal{G}(t)=0$ hold, there exist $\hat{\delta}>0$ small enough and $T>0$ such that for $t\geq T$,
\begin{align*}
w^{\prime}\leq \hat{\delta}w.
\end{align*}
Therefore  we have for $t\geq s\geq T$,
\begin{align}\label{wes2}
w(t)\leq e^{\hat{\delta}(t-s)}w(s).
\end{align}
By \eqref{Ges5} and \eqref{wes2}, we have for $t>T$
\begin{align}
\mathcal{G}(t)&\leq e^{-\delta(t-T)}+C\int_T^t e^{\delta(s-t)}\cdot \frac{1}{w(s)}ds 
\\
&\leq e^{-\delta(t-T)}+\frac{C}{w(t)}\int_T^t e^{(\delta-\hat{\delta})(s-t)}ds
\\
&\leq e^{-\delta(t-T)}+\frac{C}{\delta-\hat{\delta}}\cdot \frac{1}{w(t)}. \label{Ges7}
\end{align}
Gathering these arguments, we have
\begin{align}\label{wes3}
w^{\prime}\lesssim 1.
\end{align}
Integrating \eqref{wes3} on $[0,t]$, there exists $C_{\star}>0$ such that for all $t>0$, 
\begin{align*}
w(t)\leq C_{\star}(t+1), 
\end{align*}
which implies 
\begin{align}\label{thm13}
z(t)-\frac{1}{2}\log{t}\lesssim 1.
\end{align}
\end{proof}

\begin{proof}[Proof of Theorem \ref{centerdistance}]
By Proposition \ref{ia} and \eqref{invariance}, we complete the proof.

\end{proof}

\begin{remark}
By \eqref{Ges7}, $\mathcal{G}\lesssim e^{-2z}$ holds. Furthermore by \eqref{apmes}, \eqref{a0es}, and \eqref{Ges6}, there exists $C^{\prime}>0$ such that 
\begin{align}
\left|\frac{d}{dt}a_--\nu^-a_-\right|&\leq C^{\prime}e^{-2z}, \label{a-esref}
\\
\left|\frac{d}{dt}a_0+2\alpha a_0\right|&\leq C^{\prime}e^{-2z}. \label{a0esref}
\end{align}
By \eqref{a-esref}, we have
\begin{align}\label{a-esref2}
\left|\frac{d}{dt}\left(e^{-\nu^-t}a_-\right) \right|\leq C^{\prime}e^{-\nu^-t-2z}.
\end{align}
Integrating \eqref{a-esref2} on $(T,t)$, we obtain
\begin{align*}
\left|e^{-\nu^-t}a_-(t)-e^{-\nu^-T}a_-(T)\right|&\leq C^{\prime}\int_T^t  e^{-\nu^-s-2z(s)}ds
\\
&=C^{\prime}e^{-2z(t)}\int_T^t  e^{-\nu^-s}\cdot e^{-\hat{\delta}(s-t)}ds
\\
&\leq \frac{C^{\prime}}{-\nu^--\hat{\delta}}e^{-\nu^-t-2z(t)},
\end{align*}
since \eqref{wes2} holds. Hence we have $|a_-|\lesssim e^{-2z}$. By the same argument, we obtain $|a_0|\lesssim e^{-2z}$. Furthermore by $|a_-|+|a_0|\lesssim e^{-2z}$, \eqref{enees2}, \eqref{apes1}, and \eqref{Ges6}, we have
\begin{align*}
\mathcal{E}+\frac{\gamma}{2}|u(t,0)|^2\gtrsim \|\vec{\epsilon}\|_{\mathcal{H}}^2.
\end{align*}
Now we estimate $\mathcal{E}$. When $|u(t,0)|^2\ll e^{-2z}$, $\|\vec{\epsilon}\|_{\mathcal{H}}^2\gtrsim e^{-2z}$ holds and therefore $\mathcal{E}\gtrsim e^{-2z}$ holds. On the other hand, when $|u(t,0)|^2\gtrsim e^{-2z}$ holds, $\mathcal{E}$ satisfies $\mathcal{E}\gtrsim \|\vec{\epsilon}\|_{\mathcal{H}}^2+|u(t,0)|^2\gtrsim e^{-2z}$ since $\gamma<0$. Thus we obtain
\begin{align*}
\mathcal{E}\sim e^{-2z}.
\end{align*}
If we get the sharp estimate of $u(t,0)$, we may improve the estimate of $\|\vec{\epsilon}\|_{\mathcal{H}}$. However it is difficult to get a sharp estimate of $u(t,0)$ since a delta potential is delicate.

\end{remark}

\section*{Acknowledgement}

The author wishes to thank Kenji Nakanishi for many helpful comments, discussions and encouragement on the present paper. The author also would like to thank Nobu Kishimoto, Masaya Maeda and Takahisa Inui for their helpful discussions.

\end{document}